\documentclass[aop,MSNbibl,citesort,dvips]{arximspdf}
\usepackage{graphicx}

\doi{10.1214/10-AOP628}
\volume{40}
\issue{2}
\pubyear{2012}
\firstpage{611}
\lastpage{647}

\makeatletter

\def\bsuffix #1{#1}

\newcommand{\iint}{\int\!\!\int}

\newtheorem{theorem}{Theorem}
\newtheorem{lemma}[theorem]{Lemma}

\newtheorem{theoremm}{Theorem}
\newtheorem{variational}{Variational Problem}

\newcommand{\R}{\mathbb{R}}
\newcommand{\Z}{\mathbb{Z}}
\newcommand{\aztec}{\mbox{AD}}
\newcommand{\prob}{\mathbb{P}}
\newcommand{\pdomino}{\mathbb{P}^n_{\mbox{\fontsize{8.36pt}{8.36pt}\selectfont{\textsc{Dom}}}}}
\newcommand{\symheight}{H_{\mbox{\fontsize{8.36pt}{8.36pt}\selectfont{\textsc{Sym}}}}}
\newcommand{\punif}{\mathbb{P}_{\mbox{\fontsize{8.36pt}{8.36pt}\selectfont{\textsc{Unif}}}}}
\newcommand{\asm}{\mathcal{A}}
\newcommand{\bij}{\varphi_{\mathrm{ASM}\to\mathrm{CMT}}}
\newcommand{\feasible}{\mathcal{F}}

\makeatother

\begin{document}
\begin{frontmatter}

\title{Arctic circles, domino tilings and square~Young~tableaux}
\runtitle{Arctic circles, domino tilings and square Young tableaux}

\begin{aug}
\author[A]{\fnms{Dan} \snm{Romik}\corref{}\ead[label=e1]{romik@math.ucdavis.edu}}
\runauthor{D. Romik}
\affiliation{University of California, Davis}
\address[A]{Department of Mathematics\\
University of California\\
One Shields Ave.\\
Davis, California 95616\\
USA\\
\printead{e1}} 
\end{aug}

\received{\smonth{10} \syear{2009}}
\revised{\smonth{10} \syear{2010}}

%
\begin{abstract}
The arctic circle theorem of Jockusch, Propp, and Shor asserts that
uniformly random domino tilings of an Aztec diamond of high order are
frozen with asymptotically high probability outside the ``arctic
circle'' inscribed within the diamond. A similar arctic circle
phenomenon has
been observed in the limiting behavior of
random square Young tableaux. In this paper, we show that
random domino tilings of the Aztec diamond
are asymptotically related to random square Young tableaux in a more
refined sense that looks also at the
behavior inside the arctic circle. This is done by giving a new
derivation of the limiting shape of the height function of a random
domino tiling of the Aztec diamond that uses the large-deviation
techniques developed for the square Young tableaux problem in a previous
paper by Pittel and the author. The solution of the
variational problem that arises for domino tilings
is almost identical to the solution for the case of
square Young tableaux by Pittel and the author.
The analytic techniques used to solve the variational
problem provide a systematic, guess-free approach for solving
problems of this type which have appeared in a number of related
combinatorial probability models.
\end{abstract}

%

\end{frontmatter}

\section{Introduction}\label{intro}

\subsection{Domino tilings and the arctic circle theorem}

A domino in $\R^2$ is a~$\Z^2$-translate of either of the two sets
$[0,1]\times[0,2]$ or $[0,2]\times[0,1]$.
If $S \subset\R^2$ is a region comprised of a union of $\Z
^2$-translates of $[0,1]^2$,
a \textit{domino tiling} of $S$ is a representation of $S$
as a union of
dominoes with pairwise disjoint interiors. Domino tilings, or
equivalently the dimer model on a square lattice, are an extensively
studied and well-understood lattice model in statistical physics
and
combinatorics. Their rigorous analysis dates back to Kaste\-leyn~\cite{kasteleyn} and\vadjust{\goodbreak}
Temperley and Fisher~\cite{temperleyfisher}, who
independently derived the~for\-mula
\[
\prod_{j=1}^m \prod_{k=1}^n \biggl|2 \cos\biggl(\frac{\pi
j}{m+1}\biggr)+2\sqrt{-1} \cos\biggl(\frac{\pi k}{n+1}
\biggr)\biggr|^{1/2}
\]
for the number of domino tilings of an $n\times m$ rectangular region.
About thirty years later, a different family of regions was found to
have a much simpler formula for the number of its domino tilings: if we
define the \textit{Aztec diamond} of order $n$ to be the set
\[
\aztec_n = \bigcup_{i=-n}^{n-1} \bigcup_{j=\max
(-n-i-1,-n+i)}^{\min(n+i,n-i-1)} [i,i+1]\times[j,j+1]
\]
(see Figure \ref{fig-aztecdiamond}),
then Elkies et al. \cite{elkiesetal} proved that $\aztec_n$ has exactly
\[
2^{{n+1\choose2}}
\]
domino tilings. This can be proved by induction in several ways, but is
perhaps best understood via a connection to
\textit{alternating sign matrices}.

%
%
\begin{figure}
\begin{tabular}{@{}cc@{}}

\includegraphics{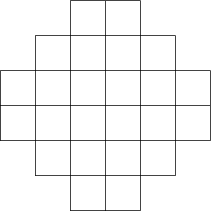}
 & \includegraphics{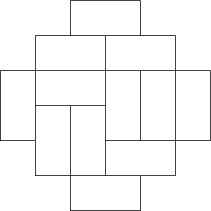}\\
(a) & (b)
\end{tabular}
\caption{The Aztec diamond of order $3$ and one of its $64$ tilings by
dominoes.}\label{fig-aztecdiamond}
\end{figure}

One of the best-known results on domino tilings is the
\textit{arctic circle theorem} due to Jockusch, Propp and Shor \cite
{jockuschproppshor}, which describes the asymptotic behavior of
uniformly random domino tilings of the Aztec diamond.
Roughly, the theorem states that
the so-called \textit{polar regions}, which are the four contiguous
regions adjacent to the four corners of the Aztec diamond in which the
tiling behaves in a predictable brickwork pattern, cover a region that
is approximately equal to the area that lies outside the circle
inscribed in the diamond. See Figure \ref{fig-arcticcircle}, where the
outline of the so-called ``arctic'' circle can be clearly discerned.
The precise statement is the following.
\begin{theorem}[(The arctic circle theorem \cite{jockuschproppshor})]
\label{thm-arcticcircle}
Fix $\varepsilon>0$. For each $n$, consider a uniformly random domino
tiling of $\aztec_n$ scaled by a factor $1/n$ in each axis to fit into
the limiting diamond
\[
\aztec_\infty:= \{ |x|+|y|\le1 \},
\]
and let $P_n^\circ\subset n^{-1}\aztec_n$ be the image of the polar
regions of the random tiling under this scaling transformation. Then as
$n\to\infty$ the event that
\begin{eqnarray*}
&&\bigl\{ (x,y)\in\aztec_\infty\dvtx x^2 + y^2 > \tfrac12 + \varepsilon
\bigr\}
\cap(n^{-1} \aztec_n)
\\
&&\qquad\subset P_n^\circ
\subset
\bigl\{ (x,y)\in\aztec_\infty\dvtx x^2 + y^2 > \tfrac12 - \varepsilon
\bigr\}
\end{eqnarray*}
holds with probability that tends to 1.
\end{theorem}

%
%
\begin{figure}

\includegraphics{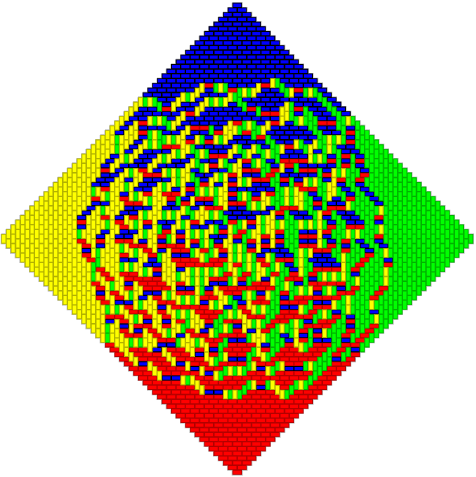}

\caption{The arctic circle theorem: in a random domino
tiling of $\aztec_{50}$, the circle-like shape is clearly visible.
Here, dominoes are colored according to their type and
parity.}\label{fig-arcticcircle}
\end{figure}

In later work, Cohn, Elkies and Propp \cite{cohnelkiespropp} derived
more detailed asymptotic information about
the behavior of random domino tilings of the Aztec diamond, that gives a
quantitative description of the behavior of the tiling inside the
arctic circle. They proved two main results (which are roughly
equivalent, if some technicalities are ignored), concerning the
\textit{placement probabilities} (the probabilities to
observe a given
type of domino in a given position in the diamond) and the
\textit{height function} of the tiling (which, roughly
speaking, encodes a
weighted counting
of the number of dominoes of different types encountered while
travelling from a fixed place to a given position in the diamond; see Section
\ref{sec-domino} for the precise definition).

A main goal of this paper is to give a new proof of the
Cohn--Elkies--Propp limit shape theorem for the height function of a
uniformly random domino tiling of the Aztec diamond; see
Theorem \ref{thm-limitshape-dom} in Section \ref{sec-domino}. Our
proof is\vadjust{\goodbreak}
based on a large deviations analysis, and so gives some information
that the proof in \cite{cohnelkiespropp} (which is based on generating
functions) does not: a large deviation principle for the
height function. Perhaps more importantly, it highlights a surprising
connection between the domino tilings model and another, seemingly
unrelated, combinatorial
probability model, namely that of \textit{random square Young
tableaux}.

%
%
\begin{figure}

\includegraphics{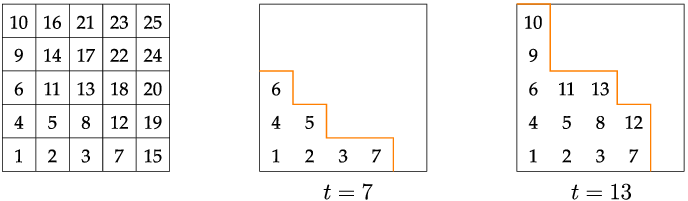}

\caption{A square Young tableau of order $5$ (shown in the
``French'' coordinate system) and the wall whose construction the
tableau encodes at various stages of its
construction.}\label{fig-squareyoungtableau}
\end{figure}

\subsection{Random square Young tableaux}

Recall that a square (standard) Young tableau of order $n$ is an array
$(t_{i,j})_{i,j=1}^n$ of integers whose entries consist of the integers
$1,2,\ldots,n^2$, each one appearing exactly once, and such that each
row and column are arranged in increasing order. One can think of a
square Young tableau as encoding a sequence of instructions for
constructing an $n\times n$ wall of square bricks leaning against the
$x$- and $y$-axes
by laying bricks sequentially,
where the rule is that each brick can be placed only in a position
which is supported from below and from the left by existing bricks or
by the axes.
In this interpretation, the number $t_{i,j}$ represents the time at
which a brick was added in position $(i,j)$;
see Figure \ref{fig-squareyoungtableau}. The number of square Young
tableaux of order $n$ is known (via the hook-length formula of
Frame--Thrall--Robinson) to be
\[
\frac{(n^2)!}{\prod_{i,j=1}^n (i+j-1)}.
\]

In \cite{pittelromik}, Boris Pittel and the author solved the problem
of finding the \textit{limiting growth profile}, or
\textit{limit shape}, of
a randomly chosen square Young tableau of high order. In other words,
the question is to find the growth profile of the square wall
``constructed in the most random way.'' This can be expressed either in
terms of the limit in probability $L(x,y)$ of the scaled tableau
entries~$n^{-2} t_{i,j}$, where $(x,y)\in[0,1]^2$ and $i=i(n)$ and
$j=j(n)$ are some sequences such that $i/n\to x$ and $j/n \to y$ as
$n\to\infty$; or alternatively in terms of the limiting shape of the
family of scaled ``sublevel sets'' $\{ n^{-1} (i,j) \dvtx t_{i,j} \le
\alpha\cdot n^2 \}$ for each $\alpha\in(0,1)$ (which in the
``wall-building'' metaphor represents the shape of the wall at various
times, and thus can be thought of as encoding the growth profile of the wall).
Figure \ref{fig-tableaulimitshape} shows the limiting growth profile
found by Pittel and Romik and the corresponding profile of a randomly
sampled square Young tableau of order $100$.

%
%
\begin{figure}

\includegraphics{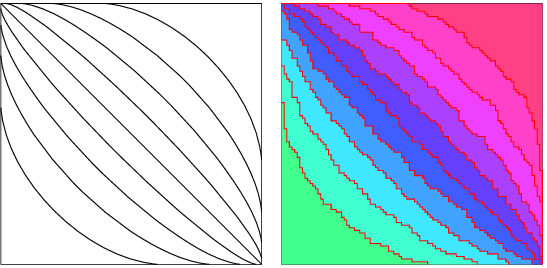}

\caption{The limiting growth profile of a random square
Young tableau and the profile of a randomly sampled tableau of order
$100$. The curves shown correspond to (scaled) times $t=i/10$,
$i=1,2,\ldots,9$.}\label{fig-tableaulimitshape}
\end{figure}

For the precise definition of the limiting growth profile, see
\cite{pittelromik}.
Here, we mention only the following fact which will be needed in the
next subsection: if $L\dvtx[0,1]\times[0,1]\to[0,1]$ is the limit shape
function mentioned above, then its values along the boundary of the
square are given by
%
%
\begin{eqnarray}\label{eq:isenough1}
L(0,t)&=& L(t,0) = \frac{1-\sqrt{1-t^2}}{2}\qquad (0\le t\le1),
\\
\label{eq:isenough2}
L(1,t) &=& L(t,1) = \frac{1+\sqrt{t(2-t)}}{2}\qquad (0\le
t\le1).
\end{eqnarray}
Also note that according to the limit shape theorem, the convergence of
$n^{-2} t_{i,j}$ to $L(i/n,j/n)$ as $n\to\infty$ is uniform in $i$
and $j$ (this follows easily from monotonicity considerations).

\subsection{An arctic circle theorem for square Young tableaux}

While it is not immediately apparent from the description of this limit
shape result, it follows as a simple corollary of it that random square
Young tableaux also exhibit an ``arctic circle''-type phenomenon. That
is, there is an equivalent way of visualizing the random tableau in
which a spatial phase transition can be seen occurring along a circular
boundary, where outside the circle the behavior is asymptotically
deterministic (the ``frozen'' phase) and inside the circle the behavior
is essentially random (the ``disordered'' or ``temperate'' phase). This
fact, overlooked at the time of publication of the paper \cite
{pittelromik}, was observed shortly afterwards by Benedek Valk\'o~\cite
{valko}. In fact, deducing the arctic circle result is easy and
requires only the facts (\ref{eq:isenough1}), (\ref{eq:isenough2})
mentioned above, which contain only a small part of the information of
the limit shape.

To see how the arctic circle appears, we consider a different encoding
of the information contained in the tableau via a system of particles
on the integer lattice~$\mathbb{Z}$. In this encoding we have $n$
particles numbered $1,2,\ldots,n$, where initially, each particle with
index $k$ is in position $k$. The particles are constrained to remain
in the interval $[1,2n]$. At discrete time steps, particles jump one
step to the right, provided that the space to their right is empty (and
provided that they do not leave the interval $[1,2n]$). At each time
step, exactly one particle jumps.

It is easy to see that after exactly $n^2$ steps, the system will
terminate when it reaches the state in which each particle $k$ is in
position $n+k$, and no further jumps can take place. We call the
instructions for evolving the system of particles from start to finish
a \textit{jump sequence}. We can now~add a~probabilistic
element to this
combinatorial model by considering the uniform probability measure on
the set of all jump sequences of order $n$, and name the resulting
probability model the \textit{jump process} of order $n$.
But in fact,
this is nothing more than a thinly disguised version of the random
square Young tableaux model, since jump sequences are in a simple
bijection with square Young tableaux: given a square tableau, think of
the sequence of numbers in row $k$ of the tableau as representing the
sequence of times during which particle $n+1-k$ jumps to the right.
This is illustrated in Figure \ref{fig-tableauparticlesbijection}. We
leave to the reader the easy verification that this gives the desired bijection.

%
%
\begin{figure}[b]

\includegraphics{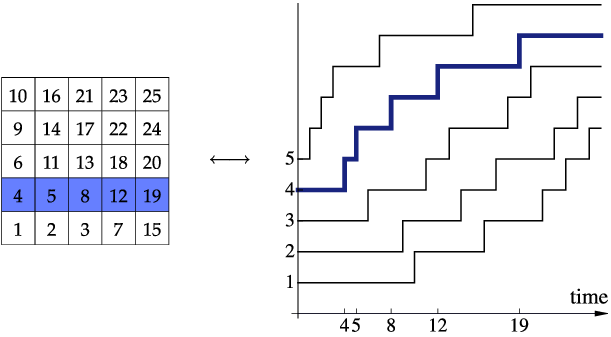}

\caption{The bijection between square Young tableaux and
jump sequences: each row in the tableau encodes the sequence of times
at which a given particle jumps. As an example, the highlighted
trajectory on the right-hand side corresponds to the highlighted row on
the left-hand side.} \label{fig-tableauparticlesbijection}
\end{figure}

With these definitions, it is now natural to consider the asymptotic
behavior of this system of particles as $n\to\infty$. Figure \ref
{fig-tableauxarctic} shows the result for a simulated system with
$n=40$. Here we see a circle-like shape appearing again. To
formulate\vadjust{\goodbreak}
precisely what is happening, given a jump process of order $n$, for
each $1\le k\le2n$, let
%
%
\begin{figure}

\includegraphics{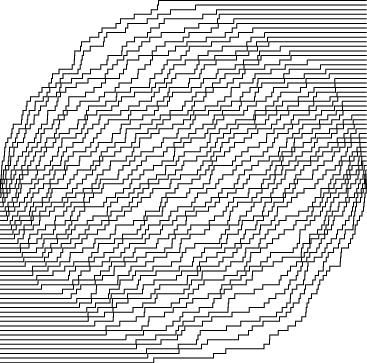}

\caption{A jump process with $40$ particles.} \label{fig-tableauxarctic}
\vspace*{-3pt}
\end{figure}
$\tau_n^{-}(k)$ and $\tau_n^{+}(k)$ denote, respectively, the first and
last times at which a particle $k$ jumped from or to position $k$.
Define the \textit{frozen time-period} in position $k$ to be
the union of
the two intervals
\[
[0,\tau_n^{-}(k)] \cup[\tau_n^{+}(k), n^2].\vspace*{-2pt}
\]

\begin{theorem}[(The arctic circle theorem for random square Young tableaux)]
\label{thm-arcticyoung}
Fix any $\varepsilon>0$.
Denote
\[
\varphi_{\pm}(x) = \tfrac12\pm\sqrt{x(1-x)}.
\]
As $n\to\infty$, the event
\begin{eqnarray*}
&&\Bigl\{
\max_{1\le k\le2n}
| n^{-2} \tau_n^{-}(k)- \varphi_{-}(k/2n)
|
< \varepsilon
\Bigr\}
\\
&&\qquad{}\cap
\Bigl\{
\max_{1\le k\le2n}
| n^{-2} \tau_n^{+}(k)- \varphi_{+}(k/2n)
|
< \varepsilon
\Bigr\}
\end{eqnarray*}
holds with probability that tends to $1$. In other words, if the
space--time diagram of the trajectories in a random jump process is
mapped to the unit square $[0,1]\times[0,1]$ by scaling the time axis
by a factor $1/n^2$ and scaling the position axis by a factor of
$1/2n$, then for large $n$ the frozen time-periods will occupy
approximately the part of the space--time diagram that lies in the
complement of the disc
\[
\{ (x,y)\in\mathbb{R}^2 \dvtx (x-1/2)^2 +
(y-1/2)^2 \le1/2 \}
\]
inscribed in the square.\vspace*{-2pt}
\end{theorem}
\begin{pf}
First, note the following simple observations that express the times
$\tau_n^{-}(k$) and $\tau_n^{+}(k)$ in terms of the Young tableau
$(t_{i,j})_{i,j=1}^n$:

\begin{longlist}
\item For $1 \le k\le n$ we have $\tau_n^{-}(k) = t_{n+1-k,1}$.\vadjust{\goodbreak}
\item For $n+1 \le k\le2n$ we have $ \tau_n^{-}(k) = t_{1,k-n}$.
\item For $1 \le k \le n$ we have $\tau_n^{+}(k) = t_{n, k}$.
\item For $n+1 \le k\le2n$ we have $\tau_n^{+}(k) = t_{2n+1-k,n}$.
\end{longlist}
For example, the first statement is based on the fact that when $1\le
k\le n$, the time $\tau_n^{-}(k)$ is simply the first time at which
the particle starting at position $k$ (which corresponds to row $n+1-k$
in the tableau) jumps. The three remaining cases are equally simple and
may be easily verified by the reader.

Combining these observations with (\ref{eq:isenough1}) and (\ref
{eq:isenough2}) and the limit shape theorem, we now see that after
scaling the times $\tau_n^{-}(k)$ and $\tau_n^{+}(k)$ by a factor of
$n^{-2}$, we get quantities that converge in the limit, uniformly in
$k$, to values determined by the appropriate substitution of boundary
values in the limit shape function $L(x,y)$. For example, to deal with
case (i) above, when $1\le k\le n$, using (\ref{eq:isenough1}) we have that
\begin{eqnarray*}
n^{-2} \tau_n^{-}(k) &=& n^{-2} t_{n+1-k,1} \approx
L\biggl(0,1-\frac{k-1}{n}\biggr)\\ &=& \frac{ 1-\sqrt{1-
(1-({k-1})/{n})^2}}{2} = \frac{1-\sqrt{({k-1})/{n}
(1-({k-1})/{n})}}{2} \\ &=&
\varphi_{-}\biggl(\frac{k-1}{2n}\biggr) \approx\varphi_{-}(k/2n),
\end{eqnarray*}
uniformly in $1 \le k \le n$. Similarly, the other three cases each
imply that $n^{-2}\tau_n^{\pm}(k)$ is uniformly
close to $\varphi_{\pm}(k/2n)$ in the appropriate range of values of
$k$; we omit the details. Combining these four cases gives exactly that
the event in Theorem \ref{thm-arcticyoung} holds with asymptotically
high probability as $n\to\infty$.
\end{pf}

\subsection{Similarity of the models and the analytic technique}

Apart from giving a new proof of the limit shape theorem of Cohn,
Elkies and Propp, another main goal of this paper is to show that the
two models described in the preceding sections (random domino tilings
of the Aztec diamond and random square Young tableaux) exhibit similar
behavior on a more detailed level than that of the mere appearance of
the arctic circle, and that in fact they are almost equivalent in an
asymptotic sense. Our new proof of the limit shape theorem for the
height function will use the same techniques developed in \cite
{pittelromik} for the case of random square Young tableaux: we first
derive a large deviations principle, not for domino tilings but for a
related model of random \textit{alternating sign matrices},
then solve the
resulting problem in the calculus of variations using an analysis that
parallels, to a remarkable (and, in our opinion, rather surprising)
level of similarity, the analysis of the variational problem in \cite
{pittelromik}. The resulting formulas for the solution of the
variational problem are almost identical to the formulas for the
limiting growth profile of random square Young tableaux. Up to some
trivial scaling factors related to the choice of coordinate system, the
formulas for the two limit shapes can be written in such a way that the
only difference between them is a single minus sign.

Another important aspect of our results lies not in the results
themselves but in the techniques used. We use the methods first
presented in~\cite{pittelromik} to solve another variational problem
belonging to a class of problems previously thought to be difficult to
analyze, due to a lack of a systematic framework that enables one to
derive the solution in a relatively mechanical way (as opposed to
having to guess it using some deep analytic insight) and then
rigorously verify its claimed extremal properties. This justifies to
some extent the claim from \cite{pittelromik} that the analytic
techniques of that paper provide a systematic approach for dealing with
such problems, which seem to appear frequently in the analysis of
combinatorial probability models (see
\cite
{cohnlarsenpropp,loganshepp,pittelromik,vershikkerov1,vershikkerov2}),
and are also
strongly related to classical variational problems arising in
electrostatics and in random matrix theory.

The rest of the paper is organized as follows. In Section
\ref{sec-asm} we recall some facts about alternating sign matrices,
and study
the problem of finding the limiting height matrix of an alternating
sign matrix chosen randomly according to \textit{domino
measure}, which
is a natural (nonuniform) probability measure on the set of
alternating sign matrices of order $n$. In Section \ref{sec-largedev}
we derive a~large deviation principle for this model. This problem is
solved in Section \ref{sec-variational-solution}.
In Section \ref{sec-limitshape-asm} we prove a limiting shape theorem
for the height matrix of an alternating sign matrix chosen according
to domino measure. In Section \ref{sec-domino} we deduce from the
previous results the Cohn--Elkies--Propp limiting shape theorem for the
height function of
uniformly random domino tilings of the Aztec
diamond. Section \ref{sec-finalremarks} has some final remarks,
including a~discussion on the potential applicability of our methods to
attack the well-known
open problem of the limit shape of uniformly random alternating sign matrices.

\section{Alternating sign matrices} \label{sec-asm}

An \textit{alternating sign matrix} (often abbreviated as
\textit{ASM}) of
order $n$ is an $n\times n$ matrix with entries in $\{0,-1,1\}$ such
that in every row and every column, the sum of the entries is $1$, and
the nonzero numbers appear with alternating signs. See Figure \ref
{fig-asmexample}(a) for an example.
Alternating sign matrices were first defined and studied in the early
1980s by Robbins and Rumsey in connection with their study
\cite{robbinsrumsey} of Dodgson's condensation method for computing
determinants and of the \textit{$\lambda$-determinant}, a natural
generalization of the determinant that arises from the condensation
algorithm. Later, Robbins, Rumsey and Mills published several
intriguing theorems and conjectures about them \cite{millsrobbinsrumsey},
tying them to the study of plane partitions and leading to many later
interesting developments, some of which are described, for example, in
\cite{bressoud,proppmanyfaces}.

%
%
\begin{figure}
\begin{tabular}{@{}cc@{}}
$  \left( \begin{array}{r@{\hspace*{11pt}}r@{\hspace*{11pt}}r@{\hspace*{11pt}}r@{\hspace*{11pt}}r@{\hspace*{11pt}}r}
0 & 0 & 1 & 0 & 0 & 0 \\
0 & 1 & -1 & 0 & 1 & 0 \\
1 & -1 & 0 & 1 & 0 & 0 \\
0 & 1 & 0 & 0 & -1 & 1  \\
0 & 0 & 1 & 0 & 0 & 0 \\
0 & 0 & 0 & 0 & 1 & 0
\end{array} \right)
$
&
$  \left(\matrix{
0 & 0 & 0 & 0 & 0 & 0 & 0 \cr
0 & 0 & 0 & 1 & 1 & 1 & 1 \cr
0 & 0 & 1 & 1 & 1 & 2 & 2 \cr
0 & 1 & 1 & 1 & 2 & 3 & 3 \cr
0 & 1 & 2 & 2 & 3 & 3 & 4  \cr
0 & 1 & 2 & 3 & 4 & 4 & 5 \cr
0 & 1 & 2 & 3 & 4 & 5 & 6}
\right)
$\vspace*{8pt}
\\
(a) & (b)
\end{tabular}
\caption{\textup{(a)} An ASM of order 6; \textup{(b)} its height matrix.}
\label{fig-asmexample}
\end{figure}

Denote by $\asm_n$ the set of ASMs of order $n$. For a matrix $M\in
\asm_n$, denote by~$N_+(M)$ the number of its entries equal to $1$. An
important formula proved by Mills, Robbins and Rumsey states that
%
%
\begin{equation}\label{eq:dominomeasure}
\sum_{M\in\asm_n} 2^{N_+(M)} = 2^{{n+1\choose2}}.
\end{equation}
This is sometimes referred to as the ``2-enumeration'' of ASMs. The
reader may note that the right-hand side is equal to the number of
domino tilings of $\aztec_n$ mentioned at the beginning of the
\hyperref[intro]{Introduction}; indeed, a combinatorial explanation for (\ref{eq:dominomeasure})
in terms of domino tilings was found by Elkies et al.
\cite{elkiesetal}. In Section \ref{sec-domino} we will say more about this
connection and how to make use of it, but for now, we rewrite (\ref
{eq:dominomeasure}) more probabilistically as
\[
2^{-{n+1\choose2}} \sum_{M\in\asm_n} 2^{N_+(M)} = 1,
\]
and consider this as the basis for defining a probability measure on
$\asm_n$, which we call \textit{domino measure} (thus named
since it is
closely related to the uniform measure on domino tilings of $\aztec
_n$; see Section \ref{sec-domino}), given by the expression
\[
\pdomino(M) = 2^{N_+(M)-{n+1\choose2}}\qquad (M\in\asm_n).
\]
Our first goal will be to study the asymptotic behavior of large random
ASMs chosen according to domino measure, and specifically the limit
shape of their \textit{height matrix}. The height matrix of
an ASM
$M=(m_{i,j})_{i,j=1}^n \in\asm_n$ is defined to be the new matrix
$H(M) = ( h_{i,j} )_{i,j=0}^n$ of order $(n+1)\times(n+1)$ whose
entries are given by
\[
h_{i,j} = \sum_{p \le i} \sum_{q \le j} m_{p,q}.
\]
The matrix $H(M)$ is also sometimes referred to as the
\textit{corner sum matrix} of~$M$.
It satisfies the following conditions:
{\setcounter{equation}{0}
\renewcommand{\theequation}{H\arabic{equation}}
\begin{eqnarray} \label{eq:h1}\qquad
&h_{0,k} = h_{k,0} = 0 \qquad\mbox{for all }0\le k\le
n,&
\\
&h_{n,k} = h_{k,n} = k \qquad\mbox{for all }0\le k\le
n,&
\\
\label{eq:h3}
&0\le
h_{i+1,j} - h_{i,j},
h_{j,i+1} - h_{j,i}\le1
\qquad\mbox{for all }0\le i<n, 0\le j\le n.&
\end{eqnarray}
}

\vspace*{-\baselineskip}

\noindent
See Figure \ref{fig-asmexample}(b) for an example.
(In fact, it is not too difficult to see that the correspondence $M\to
H(M)$ defines a bijection between the set of ASMs of order $n$ and the
set of matrices satisfying conditions (\ref{eq:h1})--(\ref{eq:h3})
(see \cite{robbinsrumsey}, Lemma 1) but we will not need this fact
here.) In particular, the ``Lipschitz''-type condition (\ref{eq:h3})
means that the height matrix can be thought of as a~discrete version of
a two-dimensional surface, and is therefore a natural candidate for
which to try and prove a limit shape result.

The basis for our analysis of $\pdomino$-random ASMs is a formula
which will give the probability distribution (under the measure
$\pdomino$) of the $k$th row of the height matrix, for each $1\le k\le
n$. To describe this, first, as usual, denote the Vandermonde function by
\[
\Delta(u_1,\ldots,u_m) = \prod_{1\le i<j \le m} (u_j-u_i).
\]
Second, for an ASM $M \in\asm_n$ and some $1\le k\le n$, let
$X_k(1)<X_k(2)<\cdots< X_k(k)$ be the unique
\textit{ascents} of the
$k$th row of the height matrix $H(M)$, namely those column indices such that
\[
h_{k,X_k(i)}-h_{k,X_k(i)-1} = 1\qquad (i=1,2,\ldots,k).
\]
Note that the conditions (\ref{eq:h1})--(\ref{eq:h3}) guarantee that
the ascents exist, that there are exactly $k$ of them and that the
original $k$th row of $H(M)$ can be recovered from them.
\begin{theorem} \label{thm-mostimportant}
If integers $1\le x_1<x_2<\cdots<x_k\le n$ are given, and if
$y_1<y_2<\cdots<y_{n-k}$ are the numbers in $\{1,2,\ldots,n\}
\setminus\{x_1,\ldots,x_k\}$ arranged in increasing order, then, in
the notation above, we have
%
\setcounter{equation}{3}
\begin{eqnarray}\label{eq:mostimportant}
&&\pdomino[ M\in\asm_n \dvtx (X_k(1),\ldots
,X_k(k))=(x_1,\ldots,x_k) ]
\nonumber\\[-8pt]\\[-8pt]
&&\qquad=
\frac{2^{{k+1\choose2}} 2^{{n-k+1\choose2}}}{2^{{n+1\choose2}}}
\cdot
\frac{\Delta(x_1,\ldots,x_k) \Delta(y_1,\ldots,y_{n-k})}{\Delta
(1,2,\ldots,k) \Delta(1,2,\ldots,n-k)}.
\nonumber
\end{eqnarray}
\end{theorem}

To prove Theorem \ref{thm-mostimportant}, we use another well-known
combinatorial bijection relating ASMs to
\textit{monotone triangles}. A
monotone triangle of order $n$
is a~triangular array
$ (t_{i,j})_{1\le i\le n, 1\le j\le i} $ of integers satisfying the inequalities
\[
t_{i,j} < t_{i,j+1},\qquad t_{i,j} \le t_{i-1,j} \le t_{i,j+1}\qquad
(2\le i\le n, 1\le j\le i-1).
\]
A \textit{complete} monotone triangle of order $n$ is a
monotone triangle
whose bottom row consists of the numbers $(1,2,\ldots,n)$. It is well
known that alternating sign matrices of order $n$ are in bijection with
complete monotone triangles of order $n$. In our terminology, the bijection
assigns to an ASM $M=(m_{i,j})_{i,j=1}^n$ the monotone triangle
\[
T= (t_{i,j})_{1\le i\le n, 1\le j\le i} =\bij(M)
\]
whose $k$th row $(t_{k,j})_{1\le j\le k}$ consists for each $1\le k\le
n$ of the ascents of the $k$th row of the height matrix $H(M)$,
arranged in increasing order. See Figure~\ref{fig-mtexample}(a) for an
%
%
\begin{figure}
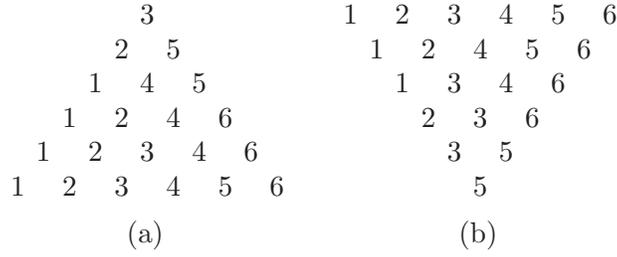

\begin{center}
\begin{tabular}{cc}
\begin{tabular}{*{15}{c@{\hspace{0.06in}}}}
& & && & 3 & & & & &
\\
& & && 2 & & 5 & & & &
\\
& && 1 & & 4 & & 5 & & &
\\
& & 1 && 2 && 4 & & 6 & &
\\
& 1 && 2 && 3 && 4 && 6 &
\\
1 && 2 && 3 && 4 && 5 && 6
\end{tabular}
&
\begin{tabular}{*{15}{c@{\hspace{0.06in}}}}
1 && 2 && 3 && 4 && 5 && 6
\\
& 1 && 2 && 4 && 5 && 6 &
\\
& & 1 && 3 && 4 & & 6 & &
\\
& && 2 & & 3 & & 6 & & &
\\
& & && 3 & & 5 & & & &
\\
& & && & 5 & & & & &
\end{tabular}\vspace*{5pt}
\\ (a) & (b)
\end{tabular}
\caption{\textup{(a)} The complete monotone triangle corresponding to the ASM
in Figure \protect\ref{fig-asmexample}; \textup{(b)}~its dual, shown
``standing on its
head.''}
\label{fig-mtexample}
\end{center}
\end{figure}
example. More explicitly, it is easy to check that this means that an
index $j$ will be present in the $k$th row of $T$ if and only if
\[
\sum_{i=1}^k m_{i,j} = 1
\]
holds.

Another notion that will prove useful is that of the
\textit{dual} of a
complete monotone triangle.
If $T$ is a complete monotone triangle of order $n$, and $M$ is the ASM
in $\asm_n$
such that $T=\bij(M)$,
then the dual $T^*$ of $T$ is the complete monotone triangle of order $n$
that corresponds via the same bijection to the matrix $W$, defined as
the vertical reflection of $M$, that is, the matrix such that
$w_{i,j}=m_{n+1-i,j}$ for all $i,j$ (clearly it, too, is an ASM). See
Figure \ref{fig-mtexample}(b), where the dual triangle is drawn
reflected vertically.

The following simple observation describes more explicitly the
connection between a monotone triangle and its dual.
\begin{lemma} \label{lem-simpleobservation}
If $T=(t_{i,j})_{1\le i\le n, 1\le j\le i} $ is a complete monotone
triangle of order $n$, then for each $1\le k\le n-1$, the $(n-k)$th row
of the dual triangle~$T^*$ consists of the numbers in the complement
\[
\{1,2,\ldots,n\} \setminus\{ t_{k,1},t_{k,2},\ldots,t_{k,k} \}
\]
of the $k$th row of $T$, arranged in increasing order.
\end{lemma}
\begin{pf} Let $M=(m_{i,j})_{i,j}\in\asm_n$ be such that $T=\bij(M)$.
As mentioned above, $1\le j\le n$ appears in the $k$th row of $T$ if
and only if
$\sum_{i=1}^k m_{i,j}=1$. Similarly, from the definition of $T^*$ we
see that $j$ appears in the $(n-k)$th row of $T^*$ if and only if
$\sum_{i=k+1}^n m_{i,j}=1$. But from the definition of an alternating
sign matrix, one and only one of these conditions must hold.
\end{pf}

As the last step in the preparation for proving Theorem \ref
{thm-mostimportant}, we note that if $M\in\asm_n$ and $T=\bij(M)$, then
it is easy to see that $N_+(M)$, the number of $+1$ entries in $M$, can
be expressed in terms of $T$ as the number of entries $t_{i,j}$ in $T$
that do not appear in the preceding row (including, vacuously, the
singleton element in the top row). We denote this quantity also by~$N_+(T)$; note that it is defined more generally also for noncomplete
monotone triangles. We furthermore recall the following formula proved
by Mills, Robbins and Rumsey in \cite{millsrobbinsrumsey}, Theorem 2,
(see also \cite{elkiesetal}, equation (7), Section~4, and see
\cite{fischeroperatorgen} for a recent alternative proof and some
generalizations):

\begin{lemma} \label{lem-vandermondecounting}
If $k\ge1$ and $x_1<x_2<\cdots<x_k$ are integers, then
the sum of~$2^{N_+(T)}$ over all monotone triangles $T$ of order $k$
with bottom row
$(x_1,\ldots,\allowbreak x_k)$ is equal to
\[
2^{{k+1\choose2}}
\prod_{1\le i<j\le k} \frac{x_j-x_i}{j-i}.
\]
\end{lemma}

\begin{pf*}{Proof of Theorem \ref{thm-mostimportant}}
Denote by $\mathcal{T}_n(x_1,\ldots,x_k)$ the set of complete
monotone triangles of order $n$ whose $k$th row is equal to
$(x_1,\ldots,x_k)$. From the remarks above, it follows that the
left-hand side of (\ref{eq:mostimportant}) is equal to
\[
2^{-{n+1\choose2}} \sum_{T\in\mathcal{T}_n(x_1,\ldots,x_k)}
2^{N_+(T)}.
\]
In addition, for a monotone triangle $T \in\mathcal{T}_n(x_1,\ldots
,x_k)$, define $T_{\mathrm{top}}$ and~$T_{\mathrm{bottom}}$ as the
two monotone triangles, of orders $k$ and $n-k$, respectively, where~%
$T_{\mathrm{top}}$ is comprised of the top $k$ rows of $T$, and
$T_{\mathrm{bottom}}$ is comprised of the top $n-k$ rows of the dual
triangle $T^*$. From Lemma \ref{lem-simpleobservation}, it follows
that the correspondence
\[
T \to(T_{\mathrm{top}}, T_{\mathrm{bottom}})
\]
defines a bijection between $\mathcal{T}_n(x_1,\ldots,x_k)$ and the
cartesian product
\mbox{$ \mathcal{A}\times\mathcal{B} $}, where $\mathcal{A}$ is the set of
monotone triangles with bottom row $(x_1,\ldots,x_k)$ and~$\mathcal
{B}$ is the set of monotone triangles with bottom row $(y_1,\ldots
,y_{n-k})$ (in the notation of Theorem~\ref{thm-mostimportant}). This
correspondence furthermore has the property that
\[
N_+(T) = N_+(T_{\mathrm{top}}) + N_+(T_{\mathrm{bottom}})
\]
[since $N_+(T_{\mathrm{top}})$ counts the number of $+1$ entries in
the first $k$ rows of the ASM corresponding to $T$, whereas
$N_+(T_{\mathrm{bottom}})$ counts the number $+1$'s in the last $n-k$ rows],
or equivalently that
$ 2^{N_+(T)} = 2^{N_+(T_{\mathrm{top}})} 2^{N_+(T_{\mathrm{bottom}})}$.
Combining these last observations, we get that the left-hand side of
(\ref{eq:mostimportant}) is equal to
\[
2^{-{n+1\choose2}}
\sum_{T_{\mathrm{top}} \in\mathcal{A}} 2^{N_+(T_{\mathrm{top}})}
\sum_{T_{\mathrm{bottom}} \in\mathcal{B}} 2^{N_+(T_{\mathrm{bottom}})},
\]
which by Lemma \ref{lem-vandermondecounting} is equal exactly to the
right-hand side of (\ref{eq:mostimportant}).
\end{pf*}

We remark that an equivalent version of
Theorem \ref{thm-mostimportant}, phrased in the language of domino
tilings and certain so-called zig--zag paths defined in terms of them,
is proved
by Johansson in \cite{johansson0} [see Proposition 5.14 in that paper
and equation (5.16) following it]. See also the subsequent papers
\cite{johansson1,johansson2} where Johansson proves many interesting
results about random domino tilings of the Aztec diamond by combining a
variant of (\ref{eq:mostimportant}) with ideas from the theory of
orthogonal polynomials and the theory of determinantal point processes.

\section{A large deviation principle} \label{sec-largedev}

We now turn from combinatorics to analysis, with the goal in mind being
to use Theorem \ref{thm-mostimportant} as the starting point for a
large deviation analysis of the behavior of $\pdomino$-random ASMs.
First, we define the space of functions on which our analysis takes
place. Fix $0< y< 1$. We wish to understand the behavior of the $k$th
row of the height matrix of a $\pdomino$-random ASM of order $n$ for
values of $k$ satisfying $k \approx y\cdot n$, when $n$ is large.

Define the space of \textit{$y$-admissible} functions to be
the set
\begin{eqnarray*}
&&\feasible_y = \{ f\dvtx[0,1]\to[0,1] \dvtx \mbox{$f$ is monotone
nondecreasing, $1$-Lipschitz,}
\\
&&\hspace*{161.1pt}\mbox{and satisfies $f(0)=0, f(1)=y$} \}.
\end{eqnarray*}
Define the space of \textit{admissible functions} as the
union of all the
$y$-admissible function spaces:
\[
\feasible= \bigcup_{y\in[0,1]} \feasible_y.
\]

We also define a discrete analogue of the admissible functions. Given
integers $0 \le k\le n$, a sequence $\mathbf{u} = (u_0,u_1,\ldots
,u_n)$ of integers
is called an \textit{$(n,k)$-admissible sequence} if it satisfies
\[
u_0=0,\qquad u_n = k\quad\mbox{and}\quad u_{i+1}-u_i \in\{0,1\}\qquad
\mbox{for all }0\le i\le n-1.
\]
Note that $(n,k)$-admissible sequences are exactly those that can
appear as the $k$th row of a height matrix $H(M)$ of an ASM $M\in\asm_n$.
We embed the $(n,k)$-admissible sequences in the space $\mathcal{F}_y$
for $y=k/n$, in the following way:
for each $(n,k)$-admissible sequence $\mathbf{u}$,
define a function $f_{\mathbf{u}} \dvtx[0,1]\to[0,1]$ as the unique
function having the values
\[
f_{\mathbf{u}}(j/n) = u_j / n,\qquad 0\le j\le n,
\]
and on each interval $[j/n, (j+1)/n]$ for $0\le j\le n-1$ being defined
as the linear interpolation of
the values on the endpoints of the interval; see Figure~\ref
{fig-usequence}. Clearly, $f_{\mathbf{u}}$ is a $(k/n)$-admissible
function. In fact, it is easy to see that the admissible functions are
precisely the limits of such functions in the uniform norm topology.

%
%
\begin{figure}

\includegraphics{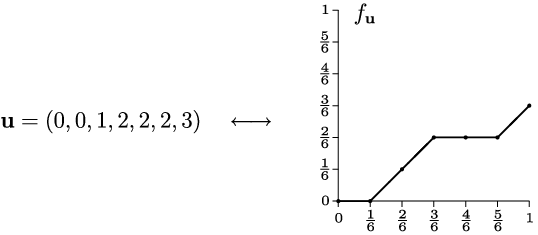}

\caption{A $(6,3)$-admissible sequence $\mathbf{u}$ and the
corresponding function
$f_{\mathbf{u}}$.}
\label{fig-usequence}
\end{figure}

With these definitions, we can now formulate the large deviation principle.

\begin{theorem}[(Large deviation principle for $\pdomino$-random ASMs)]
\label{thm-largedev}
Let $0 \le k\le n$, and let $\mathbf{u} = (u_0,u_1,\ldots,u_n)$ be an
$(n,k)$-admissible sequence.
Let $H(M)_k$ denote the $k$th row of a height matrix $H(M)$.
Then
%
%
\begin{eqnarray} \label{eq:largedevthm}
&&\pdomino[
M \in\asm_n \dvtx
H(M)_k =\mathbf{u} ] \nonumber\\[-8pt]\\[-8pt]
&&\qquad=
\exp\bigl(
-\bigl(1+o(1)\bigr) n^2 \bigl(I(f_{\mathbf{u}}) + \theta(k/n)\bigr)
\bigr),
\nonumber
\end{eqnarray}
where we define
\begin{eqnarray*}
\theta(y) &=& \frac12 y^2 \log y + \frac12 (1-y)^2 \log(1-y) +
\frac{2\log2-3}{2} y(1-y)+\frac32, \\
I(f) &=& -\int_0^1 \int_0^1 {\log}|s-t|f'(s)\bigl(f'(t)-1\bigr) \,ds
\,dt\qquad
(f\in\feasible).
\end{eqnarray*}
The $o(1)$ error term in (\ref{eq:largedevthm}) is uniform over all
$0\le k\le n$ and all
$(n,k)$-admissible sequences $\mathbf{u}$, as $n\to\infty$.
\end{theorem}
\begin{pf}
Let $1\le x_1<x_2<\cdots<x_k\le n$ be the positions of the $k$ ascents
in the sequence $(u_0,u_1,\ldots,u_n)$ (in the same sense defined
before, namely that
$u_{x_i}-u_{x_{i-1}}=1$), and let $1\le y_1<\cdots<y_{n-k}\le n$ be
the numbers in the complement
$\{1,\ldots,n\}\setminus\{x_1,\ldots,x_k\}$ arranged in increasing
order.\vadjust{\goodbreak}

By (\ref{eq:mostimportant}), we have
%
%
\begin{eqnarray} \label{eq:largedev}
&&
n^{-2} \log\pdomino[
M \in\asm_n \dvtx
H(M)_k =\mathbf{u} ]\nonumber\\[-2pt]
&&\qquad =
n^{-2} \left( \pmatrix{k+1\cr2} + \pmatrix{n-k+1\cr2} - \pmatrix
{n+1\cr2} \right) \log2
\nonumber\\[-10pt]\\[-10pt]
&&\qquad\quad{} - n^{-2} \sum_{1\le i<j \le k} \log(j-i)
- n^{-2} \sum_{1\le i<j \le n-k} \log(j-i)
\nonumber\\[-2pt]
&&\qquad\quad{}
+ n^{-2} \sum_{1\le i<j \le k} \log(x_j-x_i)
+ n^{-2} \sum_{1\le i<j \le n-k} \log(y_j-y_i).\nonumber
\end{eqnarray}
We estimate each of the summands. First, we have
%
%
\begin{eqnarray}\label{eq:estimate1}
&&n^{-2} \left( \pmatrix{k+1\cr2} + \pmatrix{n-k+1\cr2} - \pmatrix
{n+1\cr2} \right) \log2\nonumber\\[-2pt]
&&\qquad= \frac{\log2}{2} \biggl(\frac{k}{n}\biggr)^2 + \frac{\log2}{2}
\biggl(1-\frac{k}{n}\biggr)^2 - \frac{\log2}{2} + o(1)
\\[-2pt]
&&\qquad= -\log2 \cdot\frac{k}{n}\biggl(1-\frac{k}{n}\biggr) + o(1).\nonumber
\end{eqnarray}
Second, the sum $n^{-2} \sum_{1\le i<j\le k} \log(j-i)$ can be
rewritten as
%
%
\begin{eqnarray}\label{eq:estimate2}
&&n^{-2} \sum_{1\le i<j\le k} \log(j-i) \nonumber\\[-2pt]
&&\qquad= n^{-2} \sum_{d=1}^{k-1}
(k-d)\log d\nonumber\\[-2pt]
&&\qquad=
n^{-2} \sum_{d=1}^{k-1} (k-d) \log k
+
\biggl(\frac{k}{n}\biggr)^2 \sum_{d=1}^{k-1} \biggl(1-\frac
{d}{k}\biggr)\log\frac{d}{k} \cdot\frac{1}{k}
\\[-2pt]
&&\qquad=
\frac{k(k-1)}{2n^2} \log k + \biggl(\frac{k}{n}\biggr)^2 \int_0^1
(1-t)\log t \,dt + o(1)
\nonumber\\[-2pt]
&&\qquad=
\frac12 \biggl(\frac{k}{n}\biggr)^2 \log k -\frac34 \biggl(\frac
{k}{n}\biggr)^2 + o(1),\nonumber
\end{eqnarray}
where the error term $o(1)$ is uniform in $k$ as $n\to\infty$ (the
estimate for this sum is essentially the leading-order asymptotic
expansion for the so-called \textit{Barnes $G$-function},
related also to
the \textit{hyperfactorial}; for more detailed asymptotics
of these
special functions, see \cite{finch}, Section 2.15, page 135).
Similarly, replacing~$k$ by $n-k$ we get that
%
%
\begin{equation}\label{eq:estimate3}\qquad
n^{-2} \sum_{1\le i<j \le n-k} \log(j-i) = \frac12 \biggl(1-\frac
{k}{n}\biggr)^2 \log(n-k) -\frac34 \biggl(1-\frac{k}{n}\biggr)^2 +
o(1).\vadjust{\goodbreak}
\end{equation}
Finally, we estimate the terms in (\ref{eq:largedev}) that depend
directly on the sequence $\mathbf{u}$. The idea is to replace each
term $n^{-2} \log(x_j-x_i)$ with an integral of the form
$\iint\log(t-s)f_{\mathbf{u}}'(t)f_{\mathbf{u}}'(s) \,ds \,dt$ over a
certain region. Observe that for $X>1$ we have the (easily verifiable) identity
\begin{eqnarray*}
&&\int_0^1 \int_X^{X+1} \log(v-u) \,dv \,du \\
&&\qquad=\log X +
\biggl( \frac12 (X^2+1)\log\biggl(\frac{X^2-1}{X^2}\biggr) + X\log
\biggl(\frac{X+1}{X-1}\biggr)-\frac32 \biggr).
\end{eqnarray*}
When $X$ is large, this behaves like $ \log X + O(\frac
{1}{X})$. The integral is also defined and finite when $X=1$. So
we can write
\begin{eqnarray*}
&&n^{-2}\sum_{1\le i<j\le k} \log(x_j-x_i)\\
&&\qquad = n^{-2} \sum_{1\le i<j\le k} \int_{x_i-1}^{x_i} \int
_{x_j-1}^{x_j} \log(v-u) \,dv \,du
+ O\biggl(\sum_{1\le i<j\le k}
\frac{n^{-2}}{x_j-x_i}\biggr)
\\
&&\qquad=
\sum_{1\le i<j\le k} \int_{x_i-1}^{x_i} \int_{x_j-1}^{x_j} \log
(v-u) \,\frac{dv \,du}{n^2} + O\biggl(\frac{\log n}{n}\biggr)
\\
&&\qquad=
\sum_{1\le i<j\le k} \int_{x_i-1}^{x_i} \int_{x_j-1}^{x_j} \log
\biggl(\frac{v-u}{n}\biggr) \,\frac{dv \,du}{n^2}
+\log n \cdot\frac{k(k-1)}{2n^2} \\
&&\qquad\quad{} + O\biggl(\frac{\log
n}{n}\biggr).
%
\end{eqnarray*}
Now observe that $f_{\mathbf{u}}'(x)$ is equal to 1 if $(x_i-1)/n
<x<x_i/n$ for some $i$, or to $0$ otherwise; so this last expression
can be rewritten as
%
%
\begin{equation}\label{eq:theabove}
\iint_{R_n} \log(t-s) f_{\mathbf{u}}'(s) f_{\mathbf{u}}'(t) \,ds \,dt
+ \frac12 \biggl(\frac{k}{n}\biggr)^2 \log n
+ O\biggl(\frac{\log n}{n}\biggr),
\end{equation}
where the integral is over the region
\[
R_n = \bigcup_{1\le i<j\le n} \biggl[\frac{i-1}{n},\frac{i}{n}
\biggr] \times
\biggl[\frac{j-1}{n},\frac{j}{n}\biggr].
\]
The region of integration in (\ref{eq:theabove}) can be replaced with
the slightly larger region
\[
R = \{ (s,t)\in[0,1]\times[0,1] \dvtx s<t \},
\]
at the cost of an additional error which can be bounded in absolute
value by
\[
\int_0^1 dy \int_{y-1/n}^y |{\log}(y-x)| \,dx = \biggl|\int
_0^{1/n}\log t \,dt \biggr| = O\biggl(\frac{\log n}{n}\biggr).
\]
To summarize, after this change in the region of integration and, in
addition, after symmetrizing the region of integration for convenience,
we have shown that
%
%
\begin{eqnarray}\label{eq:estimate4}
n^{-2}\sum_{1\le i<j\le k} \log(x_j-x_i)
&=&
\frac12 \int_0^1 \int_0^1 {\log}|t-s| f_{\mathbf{u}}'(s) f_{\mathbf
{u}}'(t) \,ds \,dt\nonumber\\[-8pt]\\[-8pt]
&&{} + \frac12 \biggl(\frac{k}{n}\biggr)^2 \log n
+ O\biggl(\frac{\log n}{n}\biggr).
\nonumber
\end{eqnarray}
Symmetrically, following exactly the same reasoning for the last sum in
(\ref{eq:largedev}) we get the similar estimate
%
%
\begin{eqnarray}\label{eq:estimate5}\quad
&&n^{-2}\sum_{1\le i<j\le n-k} \log(y_j-y_i)
\nonumber\\
&&\qquad=\frac12 \int_0^1 \int_0^1 {\log}|t-s| \bigl(1-f_{\mathbf{u}}'(s)\bigr)
\bigl(1-f_{\mathbf{u}}'(t)\bigr) \,ds \,dt
+ \frac12 \biggl(1-\frac{k}{n}\biggr)^2 \log n\\
&&\qquad\quad{} + O\biggl(\frac{\log n}{n}\biggr).
\nonumber
\end{eqnarray}
It remains to plug the estimates (\ref{eq:estimate1}), (\ref
{eq:estimate2}), (\ref{eq:estimate3}), (\ref{eq:estimate4}) and (\ref
{eq:estimate5}) into (\ref{eq:largedev}), and simplify.
Denoting $y=k/n$, and using the integral evaluation
\[
\frac12 \int_0^1 \int_0^1 {\log}|t-s|\,ds \,dt = -\frac34,
\]
this gives that the left-hand side of (\ref{eq:largedev}) is equal to
%
%
\begin{eqnarray}
&&-\frac34 + \frac34 y^2 + \frac34 (1-y)^2 - \log2\cdot y(1-y) \nonumber\\
&&\quad{}-\frac
12 y^2 \log y - \frac12 (1-y)^2\log(1-y) \nonumber\\
&&\quad{}+ \int_0^1 \int_0^1 {\log}|t-s| f_{\mathbf{u}}'(s) f_{\mathbf{u}}'(t)
\,ds \,dt\\
&&\quad{} - \int_0^1 \int_0^1 {\log}|t-s| f_{\mathbf{u}}'(s) \,ds \,dt + o(1)
\nonumber\\
&&\qquad= -\theta(y) - I(f_{\mathbf{u}}) + o(1)
\nonumber
\end{eqnarray}
as claimed.
\end{pf}

\section{The variational problem and its solution}
\label{sec-variational-solution}

Fix $0<y<1$. Motivated by Theorem \ref{thm-largedev}, we now turn our
attention to the problem of minimizing the integral functional $I(f)$
over the appropriate class of $y$-admissible functions. In the next
section we will show how this implies a limit shape result for
$\pdomino$-random ASMs.\vadjust{\goodbreak}

The precise variational problem that we will solve is the following:
\begin{variational} \label{vari-orig}
For a given $0<y<1$, find the function $f_y^*$ that minimizes $I(f)$
over all functions $f\in\feasible_y$.
\end{variational}

Variational Problem \ref{vari-orig} is a variant of a class of
variational problems that have appeared in several random
combinatorial models (see, e.g.,
\cite{cohnlarsenpropp,loganshepp,pittelromik,vershikkerov1,vershikkerov2}).
Such problems bear a strong resemblance to classical physical problems
of finding the distribution of electrostatic charges subject to various
constraints in a one-dimensional space, as well as to problems of
finding limiting eigenvalue distributions in random matrix theory.
However, the variational problems arising from combinatorial models
usually have nonphysical constraints that make the analysis trickier.
In particular, in several of the works cited above, the presence of
such constraints required the authors to first (rather ingeniously)
\textit{guess} the solution. Once the solution was conjectured, it was
possible to verify that it is indeed the correct one using fairly
standard techniques. Cohn, Larsen and Propp, who derived the limit
shape of a random boxed plane partition, ask (see Open Question 6.3 in
\cite{cohnlarsenpropp}) whether there exists a method of solution for
their problem that does not require guessing the solution.

In \cite{pittelromik}, it was argued, however, that when dealing with
such problems, it is not necessary to guess the solution, since a
well-known formula in the theory of singular integral equations for
inverting a Hilbert transform on a finite interval actually enables
mechanically \textit{deriving} the solution rather than guessing it,
once certain intuitively plausible assumptions on the form of the
solution are made. Here, we demonstrate again the use of this more
systematic approach by using it to solve our variational problem. As an
added bonus, the solution rather elegantly turns out to be nearly
identical to the solution of the variational problem for the square
Young tableaux case (although we see no a priori reasons why this
should turn out to be the case), and we are able to make use of certain
nontrivial computations that appeared in \cite{pittelromik}, which
further simplifies the analysis.

Our goal in the rest of this section will be to prove the following theorem.

\begin{theorem}\label{thm-largedevsolution1} Define
%
%
\begin{eqnarray}\label{eq:def-zxy}
Z(x,y) &=&
\frac{2}{\pi}\biggl[(x-1/2) \arctan\biggl( \frac{\sqrt{
1/4-(x-1/2)^2-(y-1/2)^2}}{1/2-y} \biggr)
\nonumber\\
&&\hspace*{14.2pt}{} + \frac{1}{2} \arctan\biggl( \frac
{2(x-1/2)(1/2-y)}{\sqrt{1/4-(x-1/2)^2-(y-1/2)^2}} \biggr)
\\
&&\hspace*{14.2pt}{} - (1/2-y) \arctan\biggl( \frac
{x-1/2}{\sqrt{1/4-(x-1/2)^2-(y-1/2)^2}} \biggr) \biggr].
\nonumber
\end{eqnarray}
For $0<y<1/2$, the solution $f_y^*$ to Variational Problem \ref
{vari-orig} is given by
%
%
\begin{equation}\label{eq:asm-limitshape}\qquad
f_y^*(x) = \cases{
0, &\quad $\displaystyle 0 \le x \le\frac{1-2\sqrt{y(1-y)}}{2}$, \vspace*{2pt}\cr
\displaystyle \frac{y}{2}+\frac12 Z(x,y), &\quad $\displaystyle \frac{1-2\sqrt{y(1-y)}}{2} < x <
\frac{1+2\sqrt{y(1-y)}}{2}$,
\vspace*{2pt}\cr
y, &\quad $\displaystyle \frac{1+2\sqrt{y(1-y)}}{2} \le x \le1$.}
\end{equation}
For $y=1/2$, the solution is given by
\[
f_{1/2}^*(x) = \frac{x}{2}.
\]
For $y>1/2$ the solution is expressed in terms of the solution for
$1-y$ by
\[
f_y^* = x-f_{1-y}^*.
\]
Moreover, for all $0<y<1$ we have
\[
I(f_y^*) = -\theta(y).
\]
\end{theorem}

As a first step, for convenience we reformulate the variational problem
slightly to bring it to a more symmetric form, by replacing each $f \in
\feasible_y$ by the function
%
%
\begin{equation}\label{eq:coordchange}
g(x) = 2 f(x)-x.
\end{equation}
It is easy to check how the class of $y$-admissible functions
and the functional~$I(\cdot)$ transform under this mapping. The result
is the following equivalent form of our variational problem.
\begin{variational} \label{vari-reform}
For $0<y<1$, define the space of functions
\[
\mathcal{G}_y = \{ g\dvtx[0,1]\to[-1,1] \dvtx g(0) = 0,g(1) = 2y-1\mbox{,
 and $g$ is 1-Lipschitz} \}
\]
and the integral functional
\[
J(g) = -\int_0^1 \int_0^1 g'(s) {g'(t) \log}|s-t| \,ds \,dt.
\]
Find the function $g_y^* \in\mathcal{G}_y$ that minimizes the
functional $J$ over all functions $g\in\mathcal{G}_y$.
\end{variational}

The reader may verify that if $f\in\mathcal{F}_y$ and $g \in\mathcal
{G}_y$ are related by (\ref{eq:coordchange}), then the integral
functionals $I$ and $J$ are related by
\[
I(f) = \tfrac14 J(g) - \tfrac38.
\]
This implies that the following theorem is an equivalent version of
Theorem~\ref{thm-largedevsolution1}.
{\renewcommand{\thetheoremm}{7$'$}
\begin{theoremm}\label{thm-largedevsolution2}
For $0<y<1/2$, the solution $g_y^*$ to Variational Problem~\ref
{vari-reform} is given by
\[
g_y^*(x) = \cases{
-x, &\quad $\displaystyle 0 \le x \le\frac{1-2\sqrt{y(1-y)}}{2}$, \vspace*{2pt}\cr
y-x+Z(x,y), &\quad $\displaystyle \frac{1-2\sqrt{y(1-y)}}{2} < x < \frac{1+2\sqrt
{y(1-y)}}{2}$,
\vspace*{2pt}\cr
2y-x, &\quad $\displaystyle \frac{1+2\sqrt{y(1-y)}}{2} \le x \le1$,}
\]
where $Z(x,y)$ is defined in (\ref{eq:def-zxy}).
For $y=1/2$, the solution is given by $ g_{1/2}^*(x) \equiv0$.
For $y>1/2$ the solution is expressed in terms of the solution for
$1-y$ by
$ g_y^* = -g_{1-y}^*$.
Moreover, for all $0<y<1$ we have
\[
J(g_y^*) = -4\theta(y)+\tfrac32.
\]
\end{theoremm}}

\vspace*{-\baselineskip}

\begin{pf}
We now concentrate our efforts on proving
Theorem \ref{thm-largedevsolution2}. First, in the following lemma we
recall some basic facts about the space $\mathcal{G}_y$ and the
functional~$J$. We omit the proofs, since they are relatively simple
and essentially the same claims, with minor differences in the
coordinate system, were proved in \cite{pittelromik}. (See also
\cite{cohnlarsenpropp} where similar facts are proved.)
\begin{lemma}
\label{lem-alreadyproved}
\begin{longlist}
\item The space $\mathcal{G}_y$ is compact in the uniform norm.

\item The functional $J$ on $\mathcal{G}=\bigcup_{0<y<1} \mathcal{G}_y$
is a quadratic functional which can be written as
\[
J(g) = \langle g, g \rangle,
\]
where $ \langle\cdot, \cdot\rangle$ is defined by
\[
\langle g, h \rangle= -\int_0^1 \int_0^1 g'(s) {h'(t) \log}|s-t| \,ds
\,dt.
\]
The bilinear form $\langle\cdot, \cdot\rangle$ is defined for any
two Lipschitz functions $g, h$, is continuous on $\mathcal{G}$ with
the uniform norm, and is positive semidefinite in the sense that
$\langle g,g \rangle\ge0$ for any Lipschitz function $g$. The
restriction of $\langle\cdot, \cdot\rangle$ to~$\mathcal{G}_y$ is
positive-definite.

\item $J$ is strictly convex on $\mathcal{G}_y$. Therefore, a
minimizer $g_y^*$ exists and is unique.
\end{longlist}
\end{lemma}

The lemma already\vspace*{1pt} solves the problem in the case $y=1/2$, where clearly
$g_{1/2}^*\equiv0$ is the minimizer for $J$ among \textit{all}
Lipschitz functions, and in particular on $\mathcal{G}_{1/2}$. It is
also easy to see that a function $g$ is the minimizer for $J$ on
$\mathcal{G}_y$ if and only if $-g$ is the minimizer on $\mathcal
{G}_{1-y}$. So we may assume for the rest of the discussion that
$y<1/2$.\vadjust{\goodbreak}

With these preparations, we can start the analysis. We need to
minimi\-ze~%
$J(g)$ under the constraints $g \in\mathcal{G}_y$, which we rewrite as:

\begin{longlist}
\item$g(0) = 0$;
\item$g$ is differentiable almost everywhere and $g'$ satisfies
%
%
\begin{equation}\label{eq:inequalityconstraint}
-1 \le g' \le1;
\end{equation}
\item$\int_0^1 g'(x) \,dx = 2y-1$.
\end{longlist}
To address the third constraint, we consider $J$ as being defined on
the larger space $\mathcal{G}$ and form the Lagrangian
\[
\mathcal{L}(g,\lambda) = J(g) - \lambda\int_0^1 g'(x) \,dx,
\]
where $\lambda$ is a Lagrange multiplier. Minimizing $J$ under this
constraint leads, via the usual recipe for constrained optimization, to
the equation
%
%
\begin{equation}\label{eq:secondcondition}
W(s) := 
-2\int_0^1 {g'(t)\log}|s-t| \,dt - \lambda= 0.
\end{equation}
The reason for this is that, informally, $W(s)$ as defined above can be
thought of as ``the partial derivative of $\mathcal{L}$ with respect
to $g'(s)$'' [where we think of~$\mathcal{L}$ as a~function of the
uncountably many variables $(g'(s))_{s\in[0,1]}$, which is a~standard point of view in the variational calculus].

Relation (\ref{eq:secondcondition}) should hold whenever $g'(s)$ is
defined and is in $(-1,1)$. However, because of constraint (\ref
{eq:inequalityconstraint}), the condition
will be different when $g'=-1$ or \mbox{$g'=1$}. The correct condition (the
so-called ``complementary slackness'' condition) is given by the
following lemma.
\begin{lemma} \label{lem-sufficient}
If $g\in\mathcal{G}_y$ and for some real number $\lambda$ the
function $W(s)$ defined in (\ref{eq:secondcondition}) satisfies
%
%
\begin{equation} \label{eq:sufficientconds}
W(s)\mbox{ is }\cases{
= 0, &\quad if $g'(s) \in(-1,1)$, \cr
\ge0, &\quad if $g'(s) = -1$, \cr
\le0, &\quad if $g'(s) = 1$,}
\end{equation}
then $g=g_y^*$ is the minimizer for $J$ in $\mathcal{G}_y$.
\end{lemma}
\begin{pf} We copy the proof almost verbatim from \cite{pittelromik},
Lemma 7. If $h\in\mathcal{G}_y$, then in particular $h$ is
$1$-Lipschitz, so
\[
\bigl(h'(s)-g'(s)\bigr) W(s) \ge0
\]
for all $s$ for which this is defined. So
\[
\int_0^1 h'(s)W(s) \,ds \ge\int_0^1 g'(s)W(s) \,ds
\]
or in other words
\[
2 \langle g, h \rangle- \lambda(2y-1) \ge2 \langle g,g \rangle-
\lambda(2y-1),\vadjust{\goodbreak}
\]
which shows that $\langle g,h \rangle\ge\langle g,g \rangle$.
Therefore we get, using Lemma \ref{lem-alreadyproved}(ii), that
\[
\langle h,h \rangle= \langle g,g \rangle+ 2 \langle g,h-g\rangle+
\langle h-g,h-g \rangle\ge\langle g,g\rangle
\]
as claimed.
\end{pf}

Having established a \textit{sufficient} condition [comprised of the
three separate conditions in (\ref{eq:sufficientconds})] for a
function to be a minimizer, we first try to satisfy condition (\ref
{eq:secondcondition}) and save the other conditions for later. Based on
intuition that comes from the problem's connection to the combinatorial
model, we make the assumption that the minimizer $g$ is piecewise
smooth and satisfies
%
%
\begin{eqnarray}
g'(s) &\in& (-1,1) \qquad\mbox{if }s\in\biggl[
\frac{1-\beta}{2}, \frac{1+\beta}{2} \biggr], \\
\label{eq:theassumption}
g'(s) &=& -1 \qquad\mbox{if }s \notin\biggl[ \frac{1-\beta
}{2}, \frac{1+\beta}{2} \biggr],
\end{eqnarray}
where
\[
\beta= 2\sqrt{y(1-y)}.
\]
Note that $g'(s)=-1$ translates [via (\ref{eq:coordchange})] to
$f'(s)=0$ in the original space~$\mathcal{F}_y$ of $y$-admissible
functions, which corresponds to having no ascents (or very few ascents)
in the vicinity of the scaled position $(s,y)$ in the height matrix of
the ASM. Our knowledge of the endpoints of the interval in which
$g'(s)>-1$ is related to our foreknowledge of the arctic circle
theorem, and one might raise the criticism that this constitutes a
``guess.'' However, the analysis in~\cite{pittelromik} shows that it
would be possible to complete the solution even without knowing this
function in advance; here, we guess its value (which actually can be
easily guessed based on empirical evidence) so as to simplify the
analysis slightly.

Substituting this new knowledge about $g$ into (\ref
{eq:secondcondition}) gives the equation
\begin{eqnarray*}
&&
-\int_{({1-\beta})/{2}}^{({1+\beta})/{2}} {g'(t) \log}|s-t| \,dt\\
&&\qquad= \frac12 \lambda- s\log s - (1-s)\log(1-s)
+ \biggl(s-\frac
{1-\beta}{2}\biggr) \log\biggl(s-\frac{1-\beta}{2}\biggr)
\\
&&\qquad\quad{} + \biggl(\frac{1+\beta}{2}-s\biggr) \log\biggl(\frac{1+\beta
}{2}-s\biggr)
- \beta,\qquad s\in\biggl(\frac{1-\beta}{2}, \frac{1+\beta
}{2}\biggr).
\end{eqnarray*}
Differentiating with respect to $s$ then gives
%
%
\begin{eqnarray}\label{eq:hilberttrans}
-\int_{({1-\beta})/{2}}^{({1+\beta})/{2}} \frac{g'(t)}{s-t}
\,dt &=& - \log s + \log(1-s)\nonumber\\[-8pt]\\[-8pt]
&&{} + \log\biggl(s-\frac{1-\beta}{2}\biggr)
-\log\biggl(\frac{1+\beta}{2}-s\biggr).
\nonumber
\end{eqnarray}
So, just like in the analysis in \cite{pittelromik}, we have reached
the problem of inverting a Hilbert transform on a finite interval (the
so-called \textit{airfoil equation}). Moreover, the function
whose inverse
Hilbert transform we want to compute is very similar to the one that
appeared in \cite{pittelromik}---in fact, up to scaling factors only
the signs of some of the terms are permuted, and in \cite{pittelromik}
there is an extra term equal to the Lagrange multiplier $\lambda$.

Now recall that in fact the general form of the solution of equations
of this type is known. The following theorem appears in
\cite{estradakanwal}, Section
3.2, page~74 (see also~\cite{porterstirling}, Section 9.5.2):
\begin{theorem}\label{thm-airfoil}
The general solution of the airfoil equation
\[
\frac{1}{\pi} \int_{-1}^1 \frac{h(u)}{u-v} \,du = p(v),\qquad |v|<1,
\]
with the integral understood in the principal value sense, and $h$
satisfying a~H\"older condition, is given by
\[
h(v) = \frac{1}{\pi} \frac{1}{\sqrt{1-v^2}} \int_{-1}^1 \frac
{\sqrt{1-u^2} p(u)}{v-u} \,du + \frac{c}{\sqrt{1-v^2}}
\]
for some $c$.
\end{theorem}

Now set
%
%
\begin{equation}\label{eq:relation-h-p}
h(v)=g'\bigl((1+\beta v)/2\bigr).
\end{equation}
This function should satisfy
\[
\int_{-1}^1 \frac{h(u)}{u-v} \,du = \log\biggl(\frac{1-\beta
u}{2}\biggr) - \log\biggl(\frac{1+\beta u}{2}\biggr)
+\log(1+u)-\log(1-u),
\]
so, applying Theorem \ref{thm-airfoil}, we get the equation
\begin{eqnarray*}
h(v) &=& \frac{1}{\pi^2} \frac{1}{\sqrt{1-v^2}} \int_{-1}^1 \frac
{\sqrt{1-u^2}}{v-u}\biggl[ \log\biggl(\frac{1+u}{1-u}\biggr)
+\log\biggl(\frac{1-\beta u}{1+\beta u}\biggr)\biggr] \,du \\
&&{} +
\frac{c}{\sqrt{1-v^2}},
\end{eqnarray*}
where $c$ is an arbitrary constant. This can be written as
%
%
\begin{equation} \label{eq:slightlydifferent}
h(v) = \frac{1}{\pi^2\sqrt{1-v^2}}\bigl(\mathcal{I}(v,1/\beta) +
\mathcal{I}(-v,1/\beta)\bigr) + \frac{c}{\sqrt{1-v^2}},
\end{equation}
where $\mathcal{I}$ is defined by
\[
\mathcal{I}(\xi, \gamma) = \int_{-1}^1 \frac{\sqrt{1-\eta
^2}}{\xi-\eta} \log\biggl(\frac{1+\eta}{\gamma+\eta}\biggr)
\,d\eta
\]
and is evaluated in \cite{pittelromik}, Lemma 8, as
\begin{eqnarray*}
\mathcal{I}(\xi, \gamma) &=& \pi\Biggl[ 1-\gamma+\sqrt{\gamma
^2-1}-\xi\operatorname{arccosh}(\gamma)
\\
&&\hspace*{12pt}{}-2\sqrt{1-\xi^2}\arctan\sqrt{\frac{(\gamma
-1)(1-\xi)}{(\gamma+1)(1+\xi)}}
\Biggr].
\end{eqnarray*}
Therefore we get that
\begin{eqnarray*}
h(v) &=& \frac{1}{\pi\sqrt{1-v^2}} \biggl(c + \frac{\beta-1+\sqrt
{1-\beta^2}}{\beta}\biggr)
\\
&&{} - \frac{2}{\pi}\Biggl( \arctan\sqrt{\frac{(\beta
^{-1}-1)(1-v)}{(\beta^{-1}+1)(1+v)}}
+ \arctan\sqrt{\frac{(\beta^{-1}-1)(1+v)}{(\beta
^{-1}+1)(1-v)}}\Biggr).
\end{eqnarray*}
Since $c$ is an arbitrary constant, we see that the only sensible
choice that will allow $h$ to be a bounded function on the interval
$(-1,1)$ is that of $c=-(\beta-1+\sqrt{1-\beta^2})/\beta$.
So we have
\[
h(v) = - \frac{2}{\pi}\Biggl( \arctan\sqrt{\frac{(\beta
^{-1}-1)(1-v)}{(\beta^{-1}+1)(1+v)}}
+ \arctan\sqrt{\frac{(\beta^{-1}-1)(1+v)}{(\beta
^{-1}+1)(1-v)}}\Biggr).
\]

At this point, it is worth pointing out that in (\ref
{eq:slightlydifferent}), if we had the \textit{difference} of the two
$\mathcal{I}$ integrals instead of their sum, we would get at the end
(up to some trivial scaling factors that are due to the use of
different coordinate systems) exactly the function from the paper
\cite{pittelromik} that solves the variational problem for random square
Young tableaux! (Compare with equation (36) in \cite{pittelromik} and
subsequent formulas.) Thus, while the variational problems arising from
these two combinatorial models are not exactly isomorphic (which would
be perhaps less surprising), they are in some sense \textit{nearly}
equivalent. It would be interesting to understand if this phenomenon
has a conceptual explanation of some sort, but we do not see one at
present.

Simplifying the expression for $h$ using the sum-of-arctangents identity
\[
\arctan X + \arctan Y = \arctan\frac{X+Y}{1-XY}
\]
gives
\[
h(v) = -\frac{2}{\pi}\arctan\sqrt{\frac{1-\beta^2}{\beta^2-\beta
^2 v^2}}.
\]
Going back to the original function $g$ related to $h$ via (\ref
{eq:relation-h-p}), we get that
\begin{eqnarray*}
g'(s) &=& h\bigl((2s-1)/\beta\bigr) = -\frac{2}{\pi} \arctan\sqrt{\frac
{{1/4}-y(1-y)}{s(1-s)+y(1-y)-{1/4}}}
\\ &=&
-\frac{2}{\pi} \arctan\biggl( \frac{1/2-y}{\sqrt{
1/4-(y-1/2)^2-(s-1/2)^2}} \biggr)
\\ &=&
\frac{2}{\pi} \arctan\biggl( \frac{\sqrt{
1/4-(y-1/2)^2-(s-1/2)^2}}{1/2-y} \biggr) - 1
\end{eqnarray*}
for $s\in(\frac{1-\beta}{2}, \frac{1+\beta}{2})$.
From this, we can now get $g$ by integration. First, from~(\ref
{eq:theassumption}) we obtain that
\[
g(s)= -s \qquad\mbox{if } 0\le s\le\frac{1-\beta}{2}.
\]
Next,
in the interval $(\frac{1-\beta}{2},\frac{1+\beta}{2}
)$ we can integrate $g'$
using the identity
\begin{eqnarray*}
&&\int_0^t \arctan\sqrt{a-u^2} \,du \\
&&\qquad=
t \arctan\sqrt{a-t^2}
+\sqrt{1+a}\arctan\biggl( \frac{t}{\sqrt
{1+a} \sqrt{a-t^2}} \biggr)\\
&&\qquad\quad{}- \arctan\biggl( \frac{t}{\sqrt{a-t^2}}\biggr)\qquad (t^2 < a),
\end{eqnarray*}
and obtain without much difficulty that
\begin{eqnarray*}
g(s) &=&
g\biggl(\frac{1-\beta}{2}\biggr) + \int_{({1-\beta})/{2}}^s
g'(x) \,dx  \\
&=& y-s +
\frac{2}{\pi}\biggl[(s-1/2) \arctan\biggl( \frac{\sqrt{
1/4-(s-1/2)^2-(y-1/2)^2}}{1/2-y} \biggr)
\\
&&\hspace*{50.4pt}{} + \frac{1}{2} \arctan\biggl( \frac{2(s-1/2)(1/2-y)}{\sqrt
{1/4-(s-1/2)^2-(y-1/2)^2}} \biggr)
\\
&&\hspace*{50.4pt}{} - (1/2-y) \arctan\biggl( \frac{s-1/2}{\sqrt{
1/4-(s-1/2)^2-(y-1/2)^2}} \biggr) \biggr]
\end{eqnarray*}
for $s\in(\frac{1-\beta}{2},\frac{1+\beta}{2})$.

Finally, from this last equation it is easy to check that
\[
g\biggl(\frac{1+\beta}{2}\biggr) =
\lim_{s\uparrow({1+\beta})/{2}} g(s) = 2y-\frac{1+\beta}{2},
\]
so, for $s>\frac{1+\beta}{2}$, again because of (\ref
{eq:theassumption}) we get that
$g(s)= 2y-s$.
In particular,~$g$ satisfies the conditions $g(0)=0, g(1)=2y-1$, and
it is also $1$-Lipschitz, so $g\in\mathcal{G}_y$.

To summarize, we have recovered as a candidate minimizer exactly the
function from Theorem \ref{thm-largedevsolution2}. We also verified
that it is in $\mathcal{G}_y$. Furthermore, by the derivation\vadjust{\goodbreak} and the
use of Theorem~\ref{thm-airfoil}, we know that it satisfies~(\ref
{eq:hilberttrans}), or in other words that $W'(s)\equiv0$ on $
(\frac{1-\beta}{2},\frac{1+\beta}{2})$. We wanted to show
that $W(s)\equiv0$ on this interval. But looking at the definition of
$W(s)$ in~(\ref{eq:secondcondition}), we see that
we are still free to choose the Lagrange multiplier $\lambda$, which
starting from (\ref{eq:hilberttrans}) has disappeared from
the analysis! So, taking $\lambda= -2 \int_0^1 {g'(t)\log}|t-1/2| \,dt$
ensures that (\ref{eq:secondcondition}) holds
on $(\frac{1-\beta}{2},\frac{1+\beta}{2})$, which is
one of the sufficient conditions in Lemma \ref{lem-sufficient}.

All that remains to finish the proof that $g=g_y^*$ is the minimizer is
to verify the second and third conditions in (\ref
{eq:sufficientconds}), which we have not considered until now. The third
condition is irrelevant, since $g'$ is never equal to $1$, so we need
to prove that~$W(s)$, which we will now re-denote by $W(s,y)$ to
emphasize its dependence on $y$, is nonnegative when $s\notin
[\frac{1-\beta}{2},\frac{1+\beta}{2}]$. Since $g'$ is an
even function, it follows that $W(\cdot,y)$ is also even, so it is
enough to check this when $s > \frac{1+\beta}{2}$.

Once again, our argument follows closely in the footsteps of the
analogous part of the proof in \cite{pittelromik}. Fix $1/2 < s \le
1$, and let $\hat{y}=\frac{1-\sqrt{1-s^2}}{2}$, so that $\beta(\hat
{y}) = s$. We know from (\ref{eq:secondcondition}) that $W(s,\hat
{y})=0$. To finish the proof, it is enough to show that
\[
\frac{\partial W(s,y)}{\partial y} \le0 \qquad\mbox{for }0\le y \le
\hat{y}.
\]
Denote $G(x,y) = g_y^*(x)$. Then
\[
\frac{\partial W(s,y)}{\partial y} = - 2\int_0^1 {\frac{\partial^2
G(t,y)}{\partial t \,\partial y} \log}|s-t| \,dt
+ 2\int_0^1 {\frac{\partial^2 G(t,y)}{\partial t \,\partial y} \log}
|t-1/2| \,dt.
\]
A computation shows that if $t \in(\frac{1-\beta(y)}{2}, \frac
{1+\beta(y)}{2})$ then
\[
\frac{\partial^2 G(t,y)}{\partial t \,\partial y} =
\frac{\partial}{\partial y} {g_y^*}'(x) =
\frac{2}{\pi}\cdot\frac{1}{\sqrt{1/4-(x-1/2)^2-(y-1/2)^2}},
\]
and otherwise $\partial^2 G(t,y)/\partial t \,\partial y$ is clearly
$0$, so that
\[
\frac{\partial W(s,y)}{\partial y}
= \frac{4}{\pi}
\int_{(1-\beta)/2}^{(1+\beta)/2} \frac{{\log}|t-1/2|-\log
(s-t)}{\sqrt{1/4-(t-1/2)^2-(y-1/2)^2}} \,dt.
\]
Now use the two standard integral evaluations
\begin{eqnarray*}
\int_{-1}^1 \frac{{\log}|x|}{\sqrt{1-x^2}} \,dx &=& -\pi\log(2),\\
\int_{-1}^1 \frac{\log(a-x)}{\sqrt{1-x^2}} \,dx &=& \pi\log
\biggl(\frac{a+\sqrt{a^2-1}}{2}\biggr)\qquad (a>1)
\end{eqnarray*}
(see \cite{gradshteynryzhik}, equation 4.241-7, page 533, and
\cite{gradshteynryzhik}, equation 4.292-3, page 553)
to conclude that
\[
\frac{\partial W(s,y)}{\partial y} = - 4 \log\biggl( \frac{s-1/2 +
\sqrt{(s-1/2)^2-(\beta/2)^2}}{\beta/2} \biggr).\vadjust{\goodbreak}
\]
Since we assumed that $y\le\hat{y}$, or in other words that $s\ge
\frac{1+\beta(y)}{2}$, it follows that
\[
\frac{\partial W(s,y)}{\partial y} \le- 4 \log\biggl( \frac
{s-1/2}{\beta/2} \biggr) \le0
\]
as claimed.

We have proved Theorem \ref{thm-largedevsolution2} (hence also
Theorem \ref{thm-largedevsolution1}), except the claim about the value
of the integral functional $J$ at the minimizer $g_y^*$. This value
could be computed in a relatively straightforward way, as was done for
the analogous claim in \cite{pittelromik}. We omit this computation,
since, as was pointed out in \cite{pittelromik}, this can also be
proved indirectly by using the large deviation principle to conclude
that the infimum of the large deviations rate functional $I(f)+\theta
(y)$ over the space $\mathcal{F}_y$ must be equal to $0$. Therefore
the proof of Theorem \ref{thm-largedevsolution2} is complete.
\end{pf}

\section{The limit shape of $\pdomino$-random ASMs}
\label{sec-limitshape-asm}

We now apply the results from the previous sections to prove a limit
shape result for the height matrix of random ASMs chosen according to
the measure $\pdomino$.
\begin{theorem} \label{thm-limitshape-asm}
Let $F(x,y)=f_y^*(x)$, where for each $0\le y\le1$, $f_y^*$ is the
function defined in (\ref{eq:asm-limitshape}).
For each~$n$ let $M_n$ be a $\pdomino$-random ASM of or\-der~$n$, and
let $H_n = H(M_n) = (h_{i,j}^n)_{i,j=0}^n$ be its associated height matrix.
Then as $n\to\infty$ we have the convergence in probability
\[
\max_{0 \le i,j \le n} \biggl|\frac{h_{i,j}^n}{n} - F(i/n,j/n)
\biggr| \mathop{\longrightarrow}_{n\to\infty}^{\mathbb{P}}
0.
\]
\end{theorem}
\begin{pf}
Fix $\varepsilon>0$.
We want to show that
\[
A_\varepsilon^n = \biggl\{
\max_{0 \le i,j \le n} \biggl|\frac{h_{i,j}^n}{n} - F(i/n,j/n)
\biggr| > \varepsilon\biggr\}
\]
satisfies $\pdomino(A_\varepsilon^n) \to0$ as $n\to\infty$.
We start by showing a weaker statement, namely that
if $y\in(0,1)$ is given, then $\pdomino(B_{\varepsilon,y}^n)\to0$ as
$n\to\infty$, where
\[
B_{\varepsilon,y}^n =
\biggl\{\max_{0 \le j \le n} \biggl|\frac{h_{\lfloor n y \rfloor
,j}^n}{n} - F(y,j/n) \biggr| > \varepsilon/2 \biggr\}
\]
(and $\lfloor x\rfloor$ denotes as usual the integer part of a real
number $x$).
To prove this, note that
\[
B_{\varepsilon,y}^n \subseteq\bigcup_{\mathbf{u}} \bigl\{ M \in\asm
_n \dvtx H(M)_{\lfloor n y \rfloor} = \mathbf{u}\bigr\},
\]
where the union is over all $(n,k)$-admissible sequences $\mathbf{u}$
(with $k=\lfloor n y \rfloor$) such that
\[
\Vert f_{\mathbf{u}}-f_y^*\Vert_\infty= \max_{0\le x\le1} |f_{\mathbf
{u}}(x) - f_y^*(x)| > \varepsilon/2
\]
(here, $\Vert\cdot\Vert_\infty$ denotes the supremum norm on continuous
functions on $[0,1]$).
The number of such sequences is bounded by the total number of
$(n,k)$-admissible sequences, which is equal to ${n\choose k} \le2^n$
[since an $(n,k)$-admissible sequence is determined by the positions of
its $k$ ascents], and for each such~$\mathbf{u}$, by Theorem \ref
{thm-largedev} we have
\[
\pdomino\bigl( M \in\asm_n \dvtx H(M)_{\lfloor n y \rfloor} = \mathbf
{u}\bigr) \le C \exp\bigl(- \bigl(1+o(1)\bigr)c(\varepsilon,y) n^2\bigr),
\]
where $C$ is a universal constant, and
%
%
\begin{equation} \label{eq:def-cepsy}
c(\varepsilon,y) = \inf\{ I(f)+\theta(y) \dvtx f \in\mathcal{F}_y,
\Vert f-f_y^*\Vert_\infty\ge\varepsilon/2 \}.
\end{equation}
If the infimum in the definition of $c(\varepsilon,y)$ were taken over
all $f\in\mathcal{F}_y$, it would be equal to $0$ by Theorem \ref
{thm-largedevsolution1}. Note, however, that the set of \mbox{$g\in\mathcal
{G}_y$} that correspond via~(\ref{eq:coordchange}) to some $f\in
\mathcal{F}_y$ participating in the infimum in~(\ref{eq:def-cepsy})
is a closed subset (in the uniform norm topology) of $\mathcal{G}_y$
that does not contain the minimizer $g_y^*$. Therefore by Theorem \ref
{thm-largedevsolution2} and Lemma \ref{lem-alreadyproved} we get that
in fact $c(\varepsilon, y)>0$. Combining these last observations, we see
that indeed $\pdomino(B_{\varepsilon,y}^n)\to0$ as $n\to\infty$.

Next, we claim that the event $A_\varepsilon^n$ is contained in the union
of a finite number (that depends on $\varepsilon$ but not on $n$) of events
$B_{\varepsilon,y_j}^n$, so if $\pdomino(B_{\varepsilon,y}^n)\to0$ for
all $y$ then also $\pdomino(A_\varepsilon^n)\to0$. This follows because
of the Lipschitz property of the height matrix and of the limit shape
function $F$, which means that proximity to the limit at a sufficiently
dense set of values of $y$ implies proximity to the limit everywhere.
The details are simple, so we leave to the reader to check that taking
$y_j = \lfloor j \varepsilon/8 \rfloor$ for $j=1,2,\ldots,\lfloor
8/\varepsilon\rfloor$ is in fact sufficient to guarantee that
\[
A_\varepsilon^n \subset\bigcup_{j=1}^{\lfloor8/\varepsilon\rfloor}
B_{\varepsilon, y_j}^n
\]
as required.
\end{pf}

In the next section we will use a connection between uniformly random
domino tilings of the Aztec diamond and $\pdomino$-random ASMs to
prove a~limit shape theorem for the height function of the random
domino tiling. It will be helpful to consider for this purpose a
variant of the height matrix of an ASM $M$, which we call the
\textit{symmetrized height matrix} (it is sometimes referred to as
the \textit{skewed summation of $M$}).
If $M\in\asm_n$, we define this as the matrix $\symheight(M)=
(h_{i,j}^*)_{i,j=0}^n$ with entries given by
\[
h^*_{i,j} = i + j - 2 H(M)_{i,j}\qquad (M\in\asm_n, 0\le i,j\le
n),
\]
where $H(M)_{i,j}$ is the $(i,j)$th entry of the (ordinary) height
matrix of $M$. See Figure \ref{fig-symheight} for an example. The
%
%
\begin{figure}
\[
\pmatrix{
0 & 1 & 2 & 3 & 4 & 5 & 6 \cr
1 & 2 & 3 & 2 & 3 & 4 & 5 \cr
2 & 3 & 2 & 3 & 4 & 3 & 4 \cr
3 & 2 & 3 & 4 & 3 & 2 & 3 \cr
4 & 3 & 2 & 3 & 2 & 3 & 2 \cr
5 & 4 & 3 & 2 & 1 & 2 & 1 \cr
6 & 5 & 4 & 3 & 2 & 1 & 0}
\]
\caption{The symmetrized height matrix of the ASM from
Figure \protect\ref{fig-asmexample}.}
\label{fig-symheight}
\end{figure}
following theorem is an equivalent version of Theorem~\ref
{thm-limitshape-asm} formulated for these matrices.
{\renewcommand{\thetheoremm}{11$'$}
\begin{theoremm} \label{thm-limitshape-asm-sym}
Let $G(x,y)=x+y-2 F(x,y)$, where $F$ is defined in Theorem \ref
{thm-limitshape-asm}.
For each $n$ let $M_n$ be a $\pdomino$-random ASM of order $n$, and
let $H_n^* = \symheight(M_n) = ({h_{i,j}^*}^n)_{i,j=0}^n$ be its
associated symmetrized height matrix.
Then as $n\to\infty$ we have the convergence in probability
\[
\max_{0 \le i,j \le n} \biggl|\frac{{h_{i,j}^*}^n}{n} - G(i/n,j/n)
\biggr| \mathop{\longrightarrow}_{n\to\infty}^{\mathbb{P}}
0.
\]
\end{theoremm}}

We remark that it would have been possible to work with symmetrized
height matrices right from the beginning. In that case the large
deviation analysis would have lead directly to Variational Problem \ref
{vari-reform} without going first through Variational Problem \ref
{vari-orig}. [Note that the limiting symmetrized height function
$G(x,y)$ can also be written as $G(x,y)=y-g_y^*(x)$, where $g_y^*$ is
the solution to Variational Problem~\ref{vari-reform}.]

\section{Back to domino tilings} \label{sec-domino}

We now recall some basic facts from \cite{elkiesetal} about domino
tilings of the Aztec diamond $\aztec_n$, their height functions, and
their connection to alternating sign matrices and their height
matrices. This will enable us to use our previous results to reprove
the Cohn--Elkies--Propp limit shape result for the height function of a
uniformly random domino tiling of~$\aztec_n$ as $n\to\infty$.

Let $\mathcal{G}=\mathcal{G}(\aztec_n)$ be the directed graph whose
vertex set is
\[
V(\aztec_n) = \{ (i,j) \in\Z^2 \dvtx |i|+|j|\le n+1 \},
\]
and where the adjacency relations are
\[
(i_1,j_1)\to(i_2,j_2) \quad\iff\quad
\begin{array}{l}
j_1=j_2\quad\mbox{and}\quad i_1-i_2 = (-1)^{n+i_1+j_1},\\
\phantom{j_1=j_2aa}\mbox{ or } \\
\hspace*{1.7pt}i_1=i_2\quad\mbox{and}\quad j_1-j_2 = (-1)^{n+i_1+j_1+1}.
\end{array}
\]
We call $\mathcal{G}(\aztec_n)$ the \textit{Aztec diamond graph}.
Note that its adjacency structure is the standard nearest-neighbor
graph structure induced from $\mathbb{Z}^2$, where in addition edges
are directed according to a checkerboard parity rule, namely, that if a
checkerboard coloring is imposed on the squares $[n,n+1]\times[m,m+1]$
in the lattice dual to $\Z^2$, then the nearest-neighbor edges $u\to
v$ are all directed such that a traveller crossing the directed edge
will see a~black square on her left; see Figure~\ref{fig-dominoheight}(a).

Define a \textit{height function} to be any function $\eta$
on $V(\aztec
_n)$ such that for any edge $u\to v$ in $\mathcal{G}(\aztec_n)$ we have
\[
\eta(u) - \eta(v) = 1\mbox{ or }-3,
\]
and such that $\eta(u)-\eta(v)=1$ whenever $u\to v$ is one of the
boundary edges. A~height function $\eta$ on $V(\aztec_n)$ is called
\textit{normalized} if $\eta(-n,0)=0$.

It is known that any domino tiling $T$ of $\aztec_n$ determines a
unique normalized height function $\eta_T$ by the requirement that
for any directed edge $u\to v$ we have
\[
\eta_T(u) - \eta_T(v) = \cases{ -3, &\quad the segment $(u,v)$ crosses a
domino tile in
$T$,\cr
1, &\quad otherwise.}
\]
Conversely, any normalized height function $\eta$ is of the form $\eta
_T$ for some domino tiling.
See Figure \ref{fig-dominoheight}(b).

%
%
\begin{figure}
\begin{tabular}{@{}cc@{}}

\includegraphics{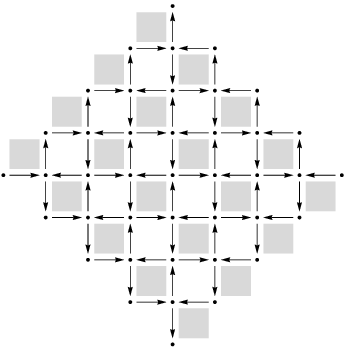}
 & \includegraphics{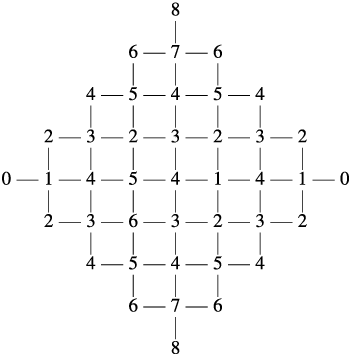}\\
(a) & (b)
\end{tabular}
\caption{\textup{(a)} The Aztec diamond graph of order 3; \textup{(b)} the normalized
height function corresponding to the tiling from Figure \protect\ref
{fig-aztecdiamond}.}
\label{fig-dominoheight}
\end{figure}

Another important fact concerns the beautiful connection, discovered by
Elkies et al. \cite{elkiesetal}, between height functions of domino
tilings of $\aztec_n$ and height matrices of ASMs: each normalized
height function $\eta$ on $V(\aztec_n)$ is essentially comprised of
the superposition of two (symmetrized) height matrices
$\symheight(A), \symheight(B)$ where $A$ is an ASM of order $n$ and
$B$ is an ASM of order $n+1$. More precisely, $\symheight(A)$ and
$\symheight(B)$ can be recovered from~$\eta$ by
%
%
\begin{eqnarray}
\label{eq:symheight-a}
\symheight(A)_{i,j} &=& \frac{\eta(-n+1+i+j,-i+j)-1}{2},\\
\label{eq:symheight-b}
\symheight(B)_{i,j} &=& \frac{\eta(-n+i+j,-i+j)}{2}
\end{eqnarray}
(note the slight difference from the formulas in \cite{elkiesetal} due
to a difference in the center of the coordinate system used).
This correspondence defines a one-to-one mapping from the set of domino
tilings of $\aztec_n$ to the set of pairs $(A,B)$ where $A\in\asm_n$
and $B\in\asm_{n+1}$. The pairs $(A,B)$ which are obtained via this
mapping are exactly the so-called \textit{compatible} pairs
defined by
Robbins and Rumsey \cite{robbinsrumsey}: $A$ and $B$ are called
compatible if
the (nonsymmetrized) height matrices $H(A), H(B)$ satisfy the conditions
\begin{eqnarray*}
H(B)_{i,j} &\le& H(A)_{i,j}, \\
H(B)_{i+1,j+1}-1 &\le& H(A)_{i,j}, \\
H(A)_{i,j} &\le& H(B)_{i+1,j}, \\
H(A)_{i,j} &\le& H(B)_{i,j+1}.
\end{eqnarray*}
It was also shown in \cite{robbinsrumsey} that for a given ASM $A \in
\asm_n$, the number of $B \in\asm_{n+1}$ that are compatible with
$A$ is equal to $2^{N_+(A)}$. Combined with the formula for the number
of domino tilings of $\aztec_n$, this implies that if $T$ is a
uniformly random domino tiling of $\aztec_n$, and $(A,B)$ is the
associated pair of compatible ASMs, then the random ASM $A$ is
distributed according to the domino measure $\pdomino$ (of course,
this provides the explanation for our choice of name for this measure).

We now combine Theorem \ref{thm-limitshape-asm-sym} with the above
discussion to easily obtain the following result, originally proved in
\cite{cohnelkiespropp}.
\begin{theorem} \label{thm-limitshape-dom}
For each $n\ge1$, let $T_n$ be a uniformly random domino tiling of
$\aztec_n$, and let $\eta_n=\eta_{T_n}$ be its associated height
function. Then as $n\to\infty$ we have the convergence in probability
\[
\max_{(i,j) \in V(\mathit{AD}_n)} \biggl| \frac{1}{n} \eta_n(i,j) -
R(i/n, j/n) \biggr| \mathop{\longrightarrow}_{n\to\infty}^{\prob} 0,
\]
where
\[
R(u,v) = 2G\biggl(\frac{u-v+1}{2}, \frac{u+v+1}{2}\biggr)\qquad
(|u| + |v| \le1),
\]
and $G$ is defined in Theorem \ref{thm-limitshape-asm-sym}.
\end{theorem}
\begin{pf} For pairs $(i,j)\in V(\aztec_n)$ for which $i+j+n$ is odd,
the proximity of $n^{-1} \eta_n(i,j)$ to $R(i/n,j/n)$ follows from
(\ref{eq:symheight-a}). For other pairs $(i,j)$, apply the previous
observation to any pair $(i',j')$ adjacent to $(i,j)$ and use the facts
that $|\eta_n(i,j)-\eta(i',j')|\le3$ and that $R$ is a continuous
function.\vadjust{\goodbreak}~%
\end{pf}

\section{Concluding remarks} \label{sec-finalremarks}

\subsection{Relation to the arctic circle theorem}

Theorem \ref{thm-limitshape-dom} implies a weak form of the arctic
circle theorem (Theorem \ref{thm-arcticcircle}): first, since inside
the arctic circle the limit shape function $R(u,v)$ is not a linear
function, it follows that the frozen region cannot extend in the limit
into the arctic circle, which is ``half'' of the theorem. In the other
direction, we get only a weaker statement that outside the arctic
circle we can have in the limit at most $o(n^2)$ ``nonfrozen''
dominoes, since that is what the linearity of the limiting height
function in that region implies.

It is interesting to contrast this with the square Young tableaux
problem. There, too, the large deviation approach gave only a bound in
one direction on the behavior of the square Young tableau along the
boundary of the square. However, Pittel and Romik managed to prove the
other direction using an additional combinatorial argument (inspired by
a method of Vershik and Kerov \cite{vershikkerov2}). It would be
interesting to see whether one can emulate this approach in the present
case to get a new proof of the arctic circle theorem. A similar
question applies to the problem of random boxed plane partitions
studied by Cohn, Larsen and Propp \cite{cohnlarsenpropp}, where again
the limit shape theorem for the height function does not imply an
arctic circle result in its strong form.

\subsection{Other arctic circles and more general arctic curves}

In this paper we have shown that two so-called arctic circle
phenomena, namely those appearing in the contexts of random domino
tilings of the Aztec diamond and of random square Young tableaux, are
closely related, in the sense that the limit shape results underlying
them can be given a more or less unified treatment using the
techniques of large deviation theory and the calculus of variations,
and that the derivations in both cases result in nearly identical
computations and formulas. Note that these are not
the only combinatorial models in which arctic circles appear. Other
examples known to the author include the shape of a uniformly random
boxed plane partition derived by Cohn, Larsen and Propp
\cite{cohnlarsenpropp} and the arctic circle theorem for random
groves, due to Petersen and Speyer \cite{randomgroves}. One might
therefore wish to extend the insights of the present paper to these
other models. The treatment
of boxed plane partitions in \cite{cohnlarsenpropp} is already based
on a large deviations analysis, and in fact the variational problem
studied there seems to be quite closely related to the variational
problems studied here and in \cite{pittelromik}. Therefore, it should be
relatively straightforward to use the techniques presented here to
give a~new derivation of the solution to the variational problem from
\cite{cohnlarsenpropp} (which in particular would provide a fully
satisfactory answer
to Open Question 6.3 from that paper).

The analysis of random groves, on the other hand, is based on
generating function techniques, and it is not clear how to apply the
ideas presented here to that setting.

It is also worth mentioning that there is a large literature on the
subject of limit shapes of various classes of random combinatorial
objects, and tiling models in particular, where one encounters in many
cases a spatial phase transition between a ``frozen'' and a
''temperate'' region. The equations governing such limit shapes can in
general lead to a much more diverse family of noncircular ``arctic
curves'' describing the shape of the interface between the frozen and
temperate regions. For details, see, for example, the papers
\cite{cohnkenyonpropp,kenyonokounkov,kenyonokounkovsheffield}.

%
%
\begin{figure}[b]

\includegraphics{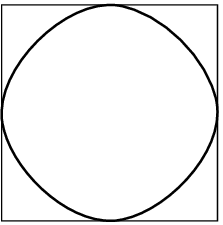}

\caption{The Colomo--Pronko conjectured limit shape for uniformly
random alternating sign matrices.}
\label{fig-colomopronko}
\end{figure}

\subsection{Uniformly random ASMs}

One reason why the methods and ideas presented in this paper may be
considered worthy of attention is somewhat speculative in nature. It
pertains to the potential future applicability of these methods and
ideas to a well-known open problem on alternating sign matrices: that
is, the problem of finding the limiting shape of a \textit{uniformly}
random ASM of high order. Here, ``limit shape'' is usually taken to
refer to the shape of the region in which the nonzero entries cluster
(the ``temperate region''), although one could also ask (as we have
done here in the case of $\pdomino$-random ASMs) about the limiting
shape of the height matrix, which also contains useful information
about the behavior of the ASM inside the temperate region.

Important progress on this question was made recently by Colomo and
Pronko~\cite{colomopronkolimitshape}, who conjectured the explicit formula
\[
x^2 + y^2 + |xy| = |x| + |y|
\]
for the limit shape of the boundary of the temperate region in a
uniformly random ASM (Figure \ref{fig-colomopronko}), and provided a
heuristic derivation of this conjectured formula based on certain
natural, but still conjectural, analytic assumptions.

In view of this state of affairs, it is worth noting that the ideas
presented in this paper seem to be rather suitable for attacking this
challenging open problem. There is only one main ``missing piece''
(albeit possibly a very substantial one) in our understanding. The idea
is to replace Theorem \ref{thm-mostimportant}, which is the
combinatorial observation which lies at the heart of the large
deviations analysis, with an analogous statement that holds for the
uniform measure on the set $\asm_n$ of ASMs of order~$n$. This
statement is given in the following theorem, whose proof follows
similar lines to the proof of Theorem~\ref{thm-mostimportant} and is omitted.
\begin{theorem}\label{thm-mostimportant-analogue}
Let $\punif$ denote the uniform measure on the set of
ASMs of order $n$.
For a positive integer $k$ and integers $x_1 < x_2 < \cdots< x_k$,
denote by $\alpha_k(x_1,\ldots,x_k)$ the number of monotone triangles
of order $k$ with bottom row $(x_1,\ldots,x_k)$.
Then, in the notation of Theorem \ref{thm-mostimportant}, we have
\begin{eqnarray*}
&&
\punif[ M\in\asm_n \dvtx (X_k(1),\ldots
,X_k(k))=(x_1,\ldots,x_k) ]
\\
&&\qquad= \frac{1}{|\asm_n|} \alpha_k(x_1,\ldots,x_k) \alpha
_{n-k}(y_1,\ldots, y_{n-k}).
\end{eqnarray*}
\end{theorem}

Unfortunately, while a formula for $|\asm_n|$ is known (see \cite{bressoud}),
the function~$\alpha_k$ seems much more difficult to understand (and
in particular, to derive asymptotics for) than the Vandermonde function
$\Delta$, and this is the piece that is missing when one tries to
duplicate our analysis to the setting of uniformly random ASMs.
Nevertheless, the function $\alpha_k$ has recently been the subject of
several very fruitful studies.
Fischer \cite{fischerformula} derived the following beautiful
``operator formula'' for $\alpha_k$:
%
%
\begin{equation}\label{eq:fischerformula}
\alpha_k(x_1,\ldots,x_k) = \biggl[ \prod_{1\le i < j \le k} (
\mathrm{Id} + E_i D_j) \biggr] \frac{\Delta(x_1,\ldots,x_k)}{\Delta
(1,\ldots,k)}.
\end{equation}
Here, $\Delta$ is the Vandermonde function as before, and $
\mathrm{Id}$, $E_j$ and $D_i$ are operators acting on the ring of polynomials
$\mathbb{C}[x_1,\ldots,x_k]$: $\mathrm{Id}$ is the identity operator,
$E_j$ is the shift operator in the variable $x_j$ (that substitutes
$x_j+1$ for each occurrence of $x_j$ in a polynomial) and
$D_i=E_i-\mathrm{Id}$ is the \mbox{(right-)differencing} operator in
the variable $x_i$.

Fischer then showed in several subsequent papers that it is possible to
use~(\ref{eq:fischerformula}) to get highly nontrivial information on
the enumeration of alternating sign matrices: in
\cite{fischerrefinedasm} she obtained a new proof of the celebrated Refined
Alternating Sign Matrix theorem (see \cite{bressoud} for the statement
and fascinating history of this result); in \cite{fischerromik} she
and the author obtained additional results concerning a
``doubly-refined'' enumeration of ASMs; and in \cite{fischer09} and~\cite{fischer10} she extended these results further to a
``multiply-refined'' enumeration. Thus, it seems quite conceivable that
additional study of $\alpha_k$ may eventually lead to a~deeper
understanding of this function,
that, in combination with Theorem~\ref{thm-mostimportant-analogue} and
the techniques of this paper, could provide a basis for a successful
attack on the limit shape problem for uniformly random ASMs.


%
\printaddresses

\end{document}